\documentclass[11pt, leqno]{amsart}

\usepackage{amssymb, amsmath, amsthm, physics}
\usepackage{comment}
\usepackage[hidelinks]{hyperref}

\numberwithin{equation}{section}
\newtheorem{thm}{Theorem}[section]
\newtheorem{prop}[thm]{Proposition}
\newtheorem{lem}[thm]{Lemma}
\newtheorem{cor}[thm]{Corollary}
\newtheorem{dfn}[thm]{Definition}

\newtheorem{rem}[thm]{Remark}

\newcommand{\R}{\mathbb{R}}
\newcommand{\C}{\mathbb{C}}
\newcommand{\Z}{\mathbb{Z}}
\newcommand{\N}{\mathbb{N}}

\newcommand{\hi}{\left}
\newcommand{\mi}{\right}
\newcommand{\bl}{\bigl}
\newcommand{\br}{\bigr}
\newcommand{\Bl}{\Bigl}
\newcommand{\Br}{\Bigr}

\newcommand{\ran}{\rangle}
\newcommand{\lan}{\langle}

\newcommand{\hs}{\hspace}

\newcommand{\mc}{\mathcal}
\newcommand{\cf}{\mathcal{F}}
\newcommand{\ch}{\mathcal{H}}
\newcommand{\ci}{\mathcal{I}}

\newcommand{\mb}{\mathbf}
\newcommand{\bfu}{\mathbf{u}}
\newcommand{\bfe}{\mathbf{E}}
\newcommand{\bff}{\mathbf{F}}
\newcommand{\bfg}{\mathbf{G}}

\newcommand{\al}{\alpha}
\newcommand{\be}{\beta}
\newcommand{\ga}{\gamma}
\newcommand{\si}{\sigma}

\newcommand{\na}{\nabla}
\newcommand{\Del}{\Delta}

\newcommand{\vep}{\varepsilon}
\newcommand{\del}{\delta}

\newcommand{\pt}{\partial_t}

\newcommand{\pa}{\partial}

\newcommand{\com}{\mathcal{C}}

\newcommand{\ove}{\overline}
\newcommand{\til}{\tilde}
\newcommand{\fr}{\frac}
\newcommand{\ope}{\operatorname}

\newcommand{\les}{\lesssim}
\newcommand{\gtr}{\gtrsim}

\newcommand{\tm}{\textrm}
\newcommand{\nt}{\notag}
\newcommand{\sm}{\setminus}
\newcommand{\cd}{\cdot}
\newcommand{\cds}{\cdots}

\allowdisplaybreaks
\title[LWP for a system of QDNLS]{Local well-posedness for a system of quadratic derivative nonlinear Schrödinger equations}
\author{Kohei Akase}
\address[K. Akase]{Department of Mathematics, Graduate School of Science, The University of Osaka, Toyonaka, Osaka, 560-0043, Japan}
\email{u306026b@ecs.osaka-u.ac.jp}

\keywords{Schr\"odinger equation;
derivative nonlinearity;
Cauchy problem;
well-posedness;
short-time Fourier restriction norm method}

\begin{document}
\subjclass[2020]{35Q55, 35B30}
\maketitle\begin{abstract}
We consider the Cauchy problem for a system of quadratic derivative nonlinear Schr\"odinger equations introduced by M. Colin and T. Colin (2004) as a model of laser-plasma interaction.
Under the condition that the flow map fails to be twice differentiable, Hirayama, Kinoshita, and Okamoto (2022) proved the well-posedness by constructing a modified energy and applying the energy method.
In the present paper, we improve the well-posedness result under the above condition by using the short-time Fourier restriction norm method.
In particular, we can remove the assumpotion of the smallness for the initial data.
\end{abstract}

\section{Introduction}
We consider the Cauchy problem for the following system of Schr\"{o}dinger equations
\begin{equation}\label{CC}
\hi\{\begin{aligned}
&(i \pt+\al \Del) u =-(\na \cd w) v, \quad(t, x) \in(0, \infty) \times \R^d, \\
&(i \pt+\be \Del) v =-(\na \cd \ove{w}) u, \quad(t, x) \in(0, \infty) \times \R^d, \\
&(i \pt+\ga \Del) w =\na(u \cd \ove{v}), \ \ \,\quad(t, x) \in(0, \infty) \times \R^d, \\
&(u(0, x), v(0, x), w(0, x)) =(u_0(x), v_0(x), w_0(x)), \quad x \in \R^d,
\end{aligned}\mi.
\end{equation}
where $\al, \be, \ga \in \R \sm\{0\}$ and the unknown functions $u, v, w$ are $\mathbb{C}^d$-valued.
For $\C^d$-valued functions $F=(F_1, \ldots, F_d)$ and $G=(G_1, \ldots, G_d)$, $F\cd G$ is defined by
\[
F\cd G = \sum_{j=1}^d F_jG_j.
\]
Also, we set $\na = (\pa_{x_1}, \ldots, \pa_{x_d})$ for $x=(x_1, \ldots, x_d)\in \R^d$, and $\na\cd F$ denotes the divergence of $F$, that is,
\[
\na\cd F =  \sum_{j=1}^d \pa_{x_j}F_j.
\]
Moreover, the initial data $(u_0, v_0, w_0)$ belongs to
\[
\ch^s(\R^d) := (H^s(\R^d))^d \times (H^s(\R^d))^d \times (H^s(\R^d))^d,
\]
where $H^{s}(\R^d)$ denotes the $L^2$-based Sobolev space. 

The system \eqref{CC} was derived by Colin and Colin in \cite{CC2004} as a model of laser-plasma interaction.
In \cite{CC2004}, Colin and Colin proved the uniqueness and local existence of the solution to \eqref{CC} in $\ch^{s}(\R^d)$ for $s>\fr{d}{2}+3$ under the assumption $\al, \be, \ga>0$.
The system \eqref{CC} has the scale invariance.
Namely, if $\bfu=(u, v, w)$ is a solution to \eqref{CC}, then $\bfu^{\lambda}(t, x) := \lambda^{-1}\bfu(\lambda^{-2}t, \lambda^{-1}x)$ is also a solution to \eqref{CC} for $\lambda>0$.
Also, the scaling critical Sobolev regularity is $s_c=\fr{d}{2}-1$.

In the case of \eqref{CC}, the relation of the coefficients $\al, \be, \ga$ of Laplacian determines the structure.
In particular, the quantity $\til\kappa:=(\al-\ga)(\be+\ga)$ plays a crucial role.
When $\til \kappa\neq 0$, Hirayama \cite{Hirayama2014}, Hirayama and Kinoshita \cite{HK2019}, Hirayama, Kinoshita, and Okamoto \cite{HKO2021} proved the well-posedness in $\ch^s(\R^d)$ for $s$ satisfying the conditions which depend on the dimension and the quantities $\kappa:=(\al-\be)(\al-\ga)(\be+\ga)$, $\mu:=\al\be\ga(\fr{1}{\al}-\fr{1}{\be}-\fr{1}{\ga})$.
They showed these results by using the Fourier restriction norm method and the Loomis-Whitney inequality.
We note that there results are almost sharp as long as we apply the iteration argument.
Also, in \cite{HKO2020}, Hirayama, Kinoshita, and Okamoto showed that the well-posedness result is improved if we consider the radial initial data in the case of $d=2, 3$ and $\mu=0$.

On the other hand, Hirayama proved in \cite{Hirayama2014} that the iteration argument does not work in $\ch^s(\R^d)$ for any $s\in\R$ in the case of $\til \kappa=0$.
In this case, Hirayama, Kinoshita, and Okamoto \cite{HKO2022} proved the well-posedness for small and large data in $\ch^{s}(\R^d)$ for $s>\fr{d}{2}+3$ if $\al, \be, \ga\in \R$ satisfy $\be+\ga\neq 0$ and $\be\ga>0$, respectively.
They proved by applying the modified energy method and the Bona-Smith approximation.
We note that the ill-posedness is shown in the case of $\be+\ga=0$, $s\in \R$ and $\al-\ga=0$, $s<0$ when we consider the $d$-dimensional torus.
See \cite{HKO2024}.

In the present paper, we improve the well-posedness result for \eqref{CC} under the condition $\til \kappa=0$, $\be+\ga\neq 0$.
The main result is the following.
\begin{thm}\label{thm1}
Let $d\in \N$, $\al, \be, \ga\in \R\sm\{0\}$ satisfy $\al-\ga=0$ and $\be + \ga \neq 0$, and 
$s> \fr{1}{2}(d+1)$.
Then, \eqref{CC} is locally well-posed in $\ch^{s}(\R^d)$.
More precisely, for any $R>0$, there exists $T\in (0, 1]$ as follows:
For any $(u_0, v_0, w_0)\in \ch^s(\R^d)$ satisfying $\|(u_0, v_0, w_0)\|_{\ch^s}\le R$, there exists a solution $(u, v, w)\in C([0, T]; \ch^s(\R^d))$ of \eqref{CC}.
Also, such solution is unique in $C([0, T]; \ch^s(\R^d))\cap \bfe^s(T)\cap \bff^s(T)$, where $\bfe^s(T)$ and $\bff^s(T)$ are defined in subsection 2.2 below.
Moreover, the data-to-solution map
\begin{align*}
S^s: \{&(u_0, v_0, w_0)\in \ch^s(\R^d) \mid \|(u_0, v_0, w_0)\|_{\ch^s}\le R\} \to C([0, T]; \ch^s(\R^d))
\end{align*}
is continuious.
\end{thm}
\noindent
The definition of the solution to \eqref{CC} is given in Definition \ref{defi sol} below.
Also, we note that we have $\bfe^{s}(T)\hookrightarrow C([0, T]; \ch^{s})$.
Moreover, Corollary \ref{str_cor} yields $\bff^{s}(T)\hookrightarrow \bfe^{s}(T)$.

To prove Theorem \ref{thm1}, we apply the short-time Fourier restriction norm method introduced by Ionescu, Kenig, and Tataru \cite{IKT}. (See also \cite{CCT}, \cite{KoTa}).
This method is the combination of the energy method and the Fourier restriction norm method.
In particular, we use the linear, nonlinear, and energy estimates.
The linear estimate is shown in \cite{IKT}.
In this paper, we show the nonlinear and energy estimates.
These are shown in Section 4 and Sections 6, 7, respectively.

We explain how we derive the exponent to be well-posed.
It is possible to obtain the local well-posedness in $\ch^{s}(\R^d)$ with $s>\fr{d}{2}+1$ by using the modified energy method with the frequency-dependent correction terms.
In this paper, we also use the bilinear Strichartz estimate.
From this, we can recover the $\fr{1}{2}$ derivative loss, and thus the exponent to be well-posed is expected to be relaxed by $\fr{1}{2}$.
This argument is performed in Section 5. 
Also, we state the optimality of the trilinear estimate under the assumption of $\al-\ga=0$ and $\be+\ga\neq 0$ in Appendix A.2.
Furthermore, the optimality of the setting of the function space defined in subsection 2.2 below is considered in Appendix A.1.

For the proof of Theorem \ref{thm1}, we first prove the existence of a smooth solution with the smooth large initial data via an approximation equation.
See \eqref{K-CC} below.
The key point to treat the large initial data is the way to construct the modified energy.
Colin and Colin \cite{CC2004} considered the system of $u, v, w$, $\pa_{x_j} u, \pa_{x_j} v, \pa_{x_j} w$, $\Del u, \Del v, \Del w$, and $\pt u, \pt v, \pt w$ and found the symmetries of this system.
Thereafter, in the case of $\be\ga>0$,  Hirayama, Kinoshita, and Okamoto \cite{HKO2022} constructed the modified energy enjoying the coercivity by using $\pt u, \pt v, \pt w$.
On the other hand, in \cite{HKO2022}, the smallness of initial data was assumed to obtain the coercivity of the modified energy in the case of $\al-\ga=0$, $\be+\ga \neq 0$.

In this paper, we use correction terms which are similar to \cite{HKO2022} but depend on the frequency.
Also, we add further correction terms to guarantee the coercivity of the modified energy also in the case of $\al-\ga=0$, and $\be+\ga\neq 0$.
From this modification, the assumption of the smallness for initial data can be removed.
See \eqref{ene10} below.

The rest of this paper is organized as follows. 
In Section 2, we prepare the notations and introduce the function spaces.
In Section 3, we give the key estimates and prove Theorem \ref{thm1} under the assumption that these key estimates hold.
In Section 4, we prove the nonlinear estimates.
In Section 5, we show the trilinear estimate which are used to prove the energy estimates.
In Section 6, we prove the energy estimate in the setting of Theorem \ref{thm1}.
In Section 7, we prove the estimate of the difference in the setting of Theorem \ref{thm1}.

\section{Preliminaries}
We first give the definition of the solution used in this paper.
\begin{dfn}\label{defi sol}
Let $d\in \N$, $s>\max\{\fr{d}{2}-1, \fr{1}{2}\}$, $T>0$.
We say $(u, v, w)\in C([0, T]; \ch^{s}(\R^d))$ is a $\ch^s$-solution to \eqref{CC} if $(u, v, w)$ satisfies
\begin{gather*}
u(t) = u_0 + i\al \int_0^t\Del u(t')\,dt' +i\int_0^t(\na\cd w)v\, dt', \\
v(t) = v_0 + i\be \int_0^t\Del v(t')\,dt' +i\int_0^t(\na\cd \ove{w})u\, dt', \\
u(t) = w_0 + i\ga \int_0^t\Del w(t')\,dt' -i\int_0^t\na(u\cd \ove{v})\, dt' 
\end{gather*}
in $\ch^{s-2}(\R^d)$ for all $t\in [0, T]$.
\end{dfn}
\noindent
We note that if $s>\max\{\fr{d}{2}-1, \fr{1}{2}\}$ and $(u, v, w)\in C([0, T]; \ch^s(\R^d))$, then we have $((\na\cd w)v, (\na\cd \ove{w})u, \na(u\cd \ove{v}))\in C([0, T]; \ch^{s-2}(\R^d))$.
Thus, the nonlinearities are well-defined in $\ch^{s-2}(\R^d)$ and we have $(u, v, w)\in C^1([0, T]; \ch^{s-2}(\R^d))$.
Moreover, the similar argument yields that $(u, v, w)\in C^2([0, T]; \ch^{s-4}(\R^d))$ if $s>\max\{\fr{d}{2}-1, \fr{3}{2}\}$.

\subsection{Notation}
We define the spatial Fourier transformation by
\[
\cf_x[f](\xi):=(2\pi)^{-\fr{d}{2}}\int_{\R^d}f(x)e^{-ix\cd\xi}\,dx.
\]
Moreover, the space-time Fourier transformation is defined by
\[
\cf_{t, x}[f](\tau, \xi):=(2\pi)^{-\fr{1}{2}(d+1)}\int_{\R}\int_{\R^d}f(t, x)e^{-ix\cd\xi-it\tau}\,dx dt.
\]
We set $\N_0:=\N\cup\{0\}$.
For multiindex $k=(k_1, \ldots, k_d)\in \N_0^d$, we set $\pa^k:=\pa_{x_1}^{k_1}\cds\pa_{x_d}^{k_d}$ and $|k|:=k_1+\cds+k_d$.
We use $A\les B$ to denote an estimate of the form $A\le CB$ for some positive constant $C$ and write $A\sim B$ to express $A\les B$ and $B\les A$.
Also, we write $A\ll B$ to mean $A\le C^{-1}B$.
For $S\subset \R^d$, we define $\mb{1}_{S}$ as a characteristic function on $S$.
For $p \in[1, \infty]$, $T>0$, and a Banach space $X$, we denote
\[
L_T^p X:=L^p([0, T] ; X).
\]
We also define $C_TX$ and so on in the same manner.
Since Section 3, with a slight abuse of notation, for a Banach space $X$ and $f=(f_1, \ldots, f_d)$, we say $f\in X$ if $f_1, \ldots, f_d\in X$.
Also, we set $\|f\|_{X}^2:=\sum_{j=1}^d\|f_j\|_{X}^2$.

In the following, we fix $\eta\in C_0^\infty(\R)$ which is even and non-increasing on $[0, \infty)$ and satisfies $\mb{1}_{[-4/3, 4/3]}\le \eta\le \mb{1}_{[-5/3, 5/3]}$.
For $j\in \N_0$, we define $\eta_j\in C_0^{\infty}(\R)$ by
\begin{equation}\label{eta}
\eta_j(\xi):=\hi\{
\begin{array}{ll}
\eta(\xi)  &(j=0), \\
\eta(\xi/2^{j})-\eta(\xi/2^{j-1})  &(j\ge 1).
\end{array}
\mi.
\end{equation}
Also, for $N\in 2^{\N_0}$, we define $\psi_N\in C_0^{\infty}(\R^d)$ by
\begin{equation*}
\psi_N(\xi):=\hi\{
\begin{array}{ll}
\eta(|\xi|)  &(N=1), \\
\eta(|\xi|/N)-\eta(2|\xi|/N)  &(N\ge 2).
\end{array}
\mi.
\end{equation*}
Then, $\psi_N$ satisfies $\sum_{N\in 2^{\N_0}} \psi_N(\xi) =1\ (\xi\in\R^d)$ and $\ope{supp}\psi_N \subset \ci_N$, where
\[
\ci_N:=\hi\{
\begin{array}{ll}
\hi\{\xi\in \R^d \mid |\xi|\le 2\mi\}&(N=1), \\
\hi\{\xi\in \R^d \mid 2^{-1}N\le |\xi|\le 2N \mi\} &(N\ge 2).
\end{array}
\mi.
\]
The (inhomogeneous) Littlewood-Paley operator is defined by $\cf_x[P_N f](\xi):=\psi_N(\xi)\cf_x[f](\xi)$.
Also, we define $P_{\le N}:= \sum_{M\in 2^{\N_0}, M\le N}P_{M}$.
For $K_1, K_2\in 2^{\N}$ satisfying $K_1\le K_2$, we define $J_{\le K_1}$ and $J_{(K_1, K_2]}$ by
\[
\cf_x[J_{\le K_1} f](\xi) := \mb{1}_{[-K_1, K_1]}(|\xi|)\cf_x[f](\xi), \quad
J_{(K_1, K_2]}:= J_{\le K_2} - J_{\le K_1}.
\]
We set $J_{\le \infty}:= Id$.
For $j\in\N_0$ and $\si\in \R$, the modulation projection $Q_j^\si$ is defined by
\[
\mc{F}_{t, x}[Q_j^\si f](\tau, \xi):= \eta_j(\tau+\si|\xi|^2)\cf_{t, x}[f](\tau, \xi).
\]

\subsection{Function spaces}

In this subsection, we define function spaces to prove the well-posedness.
For $N\in2^{\N_0}$ and $\si\in \R\sm\{0\}$, we define
\begin{align*}
X_{N, \si}=\bl\{f \in L_{\tau, \xi}^2&(\R \times \R^d) \big| f(\tau, \xi) \text { is supported on } \R\times \ci_N, \\
& \|f\|_{X_{N, \si}}:=\sum_{j\in \N_0} 2^{\fr{j}{2}}\|\eta_j(\tau+\si|\xi|^2) f(\tau, \xi)\|_{L_{\tau, \xi}^2}<\infty \br\}.
\end{align*}

The H\"older and the triangle inequalities yield 
\begin{equation}\label{Xns2}
\Bl\|\int_{\R}|f(\tau, \xi)| d \tau\Br\|_{L_\xi^2}
\les\|f\|_{X_{N, \si}}
\end{equation}
for $N \in 2^{\N_0}$, $\si\in\R\sm\{0\}$, and $f\in X_{N, \si}$.
Also, for $\si\in\R\sm\{0\}$, $N\in 2^{\N_0}$, $\ell\in \N_0$, $f \in X_{N, \si}$, and $T\in (0, 1]$, we have
\begin{align}
& \sum_{j\ge \ell + 1} 2^{\fr{j}{2}}\Bl\|\eta_j(\tau+\si|\xi|^2) \int_{\R}|f(\tau', \xi)|\cd 2^{-\ell}T(1+2^{-\ell}T|\tau-\tau'|)^{-4} d \tau'\Br\|_{L_{\tau, \xi}^2} \label{Xns3}\\
& +2^{\fr{\ell}{2}}\Bl\|\eta_{\le \ell}(\tau+\si|\xi|^2) \int_{\R}|f(\tau', \xi)|\cd 2^{-\ell}T(1+2^{-\ell}T|\tau-\tau'|)^{-4} d \tau'\Br\|_{L_{\tau, \xi}^2} \nt\\
&\les \|f\|_{X_{N, \si}}. \nt
\end{align}
Here, the implicit constant is independent of $N$, $\ell$ and $T$.
For the proof of \eqref{Xns3}, see the sentence next to (3.11) in \cite{Guo}.
From \eqref{Xns3}, for $N\in 2^{\N_0}$, $\ell \in \N_0$, $t_0 \in \R$,  $T\in (0, 1]$, $f_k \in X_k$, and $\psi \in \mc{S}(\R)$, we obtain
\begin{equation}\label{Xns4}
\|\mc{F}_{t, x}[\psi (2^{\ell}T^{-1}(t-t_0)) \cd \mc{F}_{\tau, \xi}^{-1}(f)]\|_{X_{N, \si}}
\les\|f\|_{X_{N, \si}}.
\end{equation}
In this paper, \eqref{Xns2}--\eqref{Xns4} are used to show Lemma \ref{lin}.

To measure solutions to \eqref{CC}, for $N\in 2^{\N_0}$, $\si\in \R\sm\{0\}$, and $T\in (0, 1]$, we first define $F_{N, \si}(T)$ by
\begin{align*}
F_{N, \si}(T)= \bl\{u \in C_TL_x^2\ &\big|\ \|u\|_{F_{N, \si}(T)}:=\inf _{\til{u}(t)=u(t) \text { on } [0, T]}\|\widetilde{u}\|_{F_{N, \si, T}}<\infty\br\},
\end{align*}
where the infimum is taken over $\til u \in F_{N, \si, T}$ which is defined by
\begin{equation}\label{Fnst}
\begin{aligned}
F_{N, \si, T}= \bl\{&u \in C_0(\R; L_x^2(\R^d))\ \big|\ \mc{F}_{t, x}[u](\tau, \xi) \text { is supported in } \R \times \ci_N, \\
&\|u\|_{F_{N, \si, T}}=\sup_{t_N \in \R}\|\cf_{t, x}[\eta_0(NT^{-1}(t-t_N)) \cd u]\|_{X_{N, \si}}<\infty\br\}, 
\end{aligned}
\end{equation}
\[
C_0(\R; L^2(\R^d)):=\{f\in C(\R; L^2(\R^d)) \mid f(t) = 0 \textrm{ if  } t\in \R\sm[-4, 4]\}
\]
Also, to measure nonlinearities, for $N\in 2^{\N_0}$, $\si\in \R\sm\{0\}$, and $T\in (0, 1]$, we define $G_{N, \si}(T)$ by
\begin{align*}
G_{N, \si}(T)= \bl\{f \in C_TL_x^2\ &\Big|\ \|f\|_{G_{N, \si}(T)}:=\inf _{\til{f}(t)=f(t) \text { on } [0, T]}\|\widetilde{f}\|_{G_{N, \si, T}}<\infty\br\},
\end{align*}
where the infimum is taken over $\til u \in G_{N, \si, T}$ which is defined by
\begin{align*}
&G_{N, \si, T}= \bl\{f \in C_0(\R; L_x^2(\R^d))\ \Big|\ \mc{F}_{t, x}[u](\tau, \xi) \text { is supported in } \R \times \ci_N, \\
&\|f\|_{G_{N, \si, T}}\hs{-1mm}:=\hs{-1mm}\sup_{t_N \in \R}\|(\tau + \si|\xi|^2+iNT^{-1})^{-1} \cf[\eta_0(NT^{-1}(t-t_N)) \cd f]\|_{X_{N, \si}}\hs{-1mm}<\infty \hs{-0.5mm}\br\}\hs{-0.5mm} .
\end{align*}

Now, we define the function spaces to prove the well-posedness for \eqref{CC}.
For $s\in\R$, $\si\in\R\sm\{0\}$, $t_N \in \R$, and $T\in (0, 1]$, we define $F_{\si}^s(T)$ and  $G_{\si}^s(T)$ by
\begin{gather*}
F_{\si}^s(T):=\bl\{u \in C_TH_x^s\ \bigl|\ \|u\|_{F_{\si}^s(T)}:=\bigl\|N^{s}\|P_N u\|_{F_{N, \si}(T)}\bigr\|_{\ell_N^2}<\infty\br\}, \\
G_{\si}^s(T):=\bl\{f \in C_TH_x^s\ \bigl|\ \|f\|_{G_{\si}^s(T)}:=\bigl\|N^{s}\|P_N f\|_{G_{N, \si}(T)}\bigr\|_{\ell_N^2}<\infty\br\}.
\end{gather*}
Also, we use the following energy space defined  by
\[
E^s(T):=\bl\{u \in C_TH_x^s\ \big|\ \|u\|_{E^s(T)}:= \bl\|N^{s}\|P_N u\|_{L_T^\infty L_x^2}\br\|_{\ell_N^2}\ <\infty\br\}
\]
for $s\in \R$ and $T\in (0, 1]$.
From the definition, we immediately obtain $E^s(T)\hookrightarrow C_TH^s$.
Also, we have $F_{\si}^s(T)\hookrightarrow E^s(T)$.
See Corollary \ref{str_cor} below.

For the notation, we denote
\[
\begin{array}{c}
\bfe^s(T):=  (E^s(T))^d\times (E^s(T))^d\times (E^s(T))^d, \\
\bff^s(T):=  (F_{\al}^s(T))^d\times (F_{\be}^s(T))^d\times (F_{\ga}^s(T))^d, \\
\bfg^s(T):=  (G_{\al}^s(T))^d\times (G_{\be}^s(T))^d\times (G_{\ga}^s(T))^d.
\end{array}
\]
\begin{rem}\label{rem_del}\rm
From the definitios of $F_{N, \si}(T)$, we have 
\begin{equation*}
\|\pa^k |\na|^\vep P_{\le N}u\|_{F_{N', \si}(T)}\les N^{|k|+\vep}\|u\|_{F_{N', \si}(T)}
\end{equation*}
for $N, N'\in 2^{\N_0}$, $k\in \N_0^d$, and $\vep\ge 0$.
It also holds $\||\na|^{-\vep} J_{(K_1, K_2]}u\|_{F_{N', \si}(T)}\les K_1^{-\vep}\|u\|_{F_{N', \si}(T)}$ for $K_1\le K_2\in 2^{\N}\cup \{\infty\}$ and $\vep>0$.
Moreover, we can obtain the similar estimates if we change ${F_{N', \si}(T)}$ into ${G_{N', \si}(T)}$, $F_{N', \si, T}$, and $G_{N', \si, T}$.
\end{rem}

For $N\in 2^{\N_0}$, we define the set $S_N$ by $N$-acceptable time multiplication factors: 
\[
S_N=:\bl\{m: \R \to \R \ \big|\  \|m\|_{S_N} := \sum_{0\le |k|\le 10} N^{-k} \|\pa^k m\|_{L^{\infty}(\R)}<\infty \br\}.
\]
Then, the following properties hold.

\begin{lem}
Let $N\in 2^{\N_0}$, $\si\in \R\sm\{0\}$, and $T\in (0, 1]$.
Then, we have
\begin{equation}
\|m(t) f\|_{G_{N, \si, T}} \les\|m\|_{S_N}\|f\|_{G_{N, \si, T}} \nt
\end{equation}
for any $f\in G_{N, \si, T}$ and $m\in S_N$.
\end{lem}
\begin{proof}
See the paragraph which follows (3.13) in \cite{Guo}.
\end{proof}
\begin{rem}\rm
In this paper, this lemma is used to show Lemma \ref{lin}.
Also, the same estimate holds if we change $G_{N, \si, T}$ into $F_{N, \si, T}$.
\end{rem}

\section{Proof of the main theorem}
To prove Theorem \ref{thm1}, we first show the existence of a solution to \eqref{CC} with smooth initial data.
To prove this, we consider the Cauchy problem for the approximation system
\begin{equation}\label{K-CC}
\hi\{\begin{aligned}
&(i \pt+\al \Del) u_{K} =-J_{\le K}((\na \cd J_{\le K}w_{K}) J_{\le K}v_{K}), \hs{14mm}(t, x) \in(0, \infty) \times \R^d, \\
&(i \pt+\be \Del) v_{K} =-J_{\le K}((\ove{\na \cd J_{\le K} w_{K}}) J_{\le K} u), \hs{17mm}(t, x) \in(0, \infty) \times \R^d, \\
&(i \pt+\ga \Del) w_{K} =\na J_{\le K}(J_{\le K} u_{K} \cd \ove{J_{\le K} v_{K}}), \hs{20.5mm}(t, x) \in(0, \infty) \times \R^d, \\
&(u_{K}(0, x), v_{K}(0, x), w_{K}(0, x)) =(J_{\le K}u_0(x), J_{\le K}v_0(x), J_{\le K}w_0(x)), \hs{2mm}  x \in \R^d
\end{aligned}\mi.
\end{equation}
for $K\in 2^{\N}\cup\{\infty\}$.
The system \eqref{K-CC} is equal to \eqref{CC} if $K=\infty$.
We construct a solution to \eqref{CC} by taking the limit of \eqref{K-CC} as $K\to\infty$.
It is easy to prove the local well-posedness of \eqref{K-CC} if $K\in 2^{\N}$.
\begin{lem}\label{wp_K-CC}
Let $s>\fr{d}{2}$ and $K\in 2^{\N}$.
Then, \eqref{K-CC} is locally well-posed in $\ch^{s}(\R^d)$.
Moreover, if $(u, v, w)\in C([0, T]; \ch^s(\R^d))$ is a solution to \eqref{K-CC}, then we have $(u, v, w)\in C([0, T]; \ch^{s'}(\R^d))$ for any ${s'}\in \R$.
\end{lem}
\begin{proof}

This lemma is shown by the standard iteration argument.
Thus, we give the outline of the proof.
We express \eqref{K-CC} as the corresponding integral system
\begin{gather*}
u_{K}(t)=e^{it\al\Del}J_{\le K}u_0 +i\int_0^t e^{i(t-t')\al\Del}J_{\le K}((\na \cd w_{K}(t'))v_{K}(t'))\,dt',\\
v_{K}(t)=e^{it\be\Del}J_{\le K}v_0 +i\int_0^t e^{i(t-t')\be\Del}J_{\le K}((\ove{\na \cd w_{K}(t')}) u_{K}(t'))\,dt', \\
w_{K}(t)=e^{it\ga\Del}J_{\le K}w_0 -i\int_0^t e^{i(t-t')\ga\Del}\na J_{\le K}(u_{K}(t')\cd \ove{v_{K}(t')})\,dt'.
\end{gather*}
Then, it follows from the inequalities $\|fg\|_{H^{s-1}}\les \|f\|_{H^{s-1}}\|g\|_{H^s}$ for $s>\fr{d}{2}$ and $\|\pa_{x_j} J_{\le K}f\|_{H^s}\les K\|f\|_{H^s}$ that
\begin{align*}
&\Bl\|\int_0^t e^{i(t-t')\al\Del}J_{\le K}((\na \cd J_{\le K}w_{K}(t')) J_{\le K}v_{K}(t'))\,dt'\Bl\|_{L_T^\infty H_x^s} \\
&\les HT\|w_{K}\|_{L_T^\infty H_x^s}\|v_{K}\|_{L_T^\infty H_x^s}.
\end{align*}
We have the estimate for $v, w$ in the same manner.
Thus, we have
\[
\|(u_{K}, v_{K}, w_{K})\|_{L_T^\infty \ch_x^s}
\les \|(u_0, v_0, w_0)\|_{\ch^{s}} + KT\|(u_{K}, v_{K}, w_{K})\|_{L_T^\infty \ch_x^s}^2.
\]
A similar argument yields the estimate for the difference.
By taking $T=T(\|(u_0, v_0, w_0)\|_{\ch^{s}}, K)>0$ sufficiently small, we can apply the Banach fixed-point theorem.
In particular, there exist $T_0>0$ and a solution $(u_{K}, v_{K}, w_{K})\in C([0, T_0]; \ch^s(\R^d))$ to \eqref{K-CC} satisfying
\begin{equation*}
\|(u_{K}, v_{K}, w_{K})\|_{L_t^\infty([0, T_0]; \ch^{s})}
\les \|(u_0, v_0, w_0)\|_{\ch^{s}}.
\end{equation*}
We omit the remaining part of the proof because it is a standard argument.
We note that the smoothness of the solution is obtained by the expression of the Duhamel formula and the presence of the operator $J_{\le K}$.
\end{proof}
\noindent
If $K<\infty$, then Lemma \ref{wp_K-CC} yields $(u_K, v_K, w_K)=(J_{\le K}u_K, J_{\le K}v_K, J_{\le K}w_K)$ and
\begin{equation}\label{K-CC Hs-sol}
\hi\{\hs{-1.5mm}
\begin{aligned}
u_{K}(t)& \hs{-0.5mm} = \hs{-0.5mm} J_{\le K}u_0 \hs{-0.5mm} + \hs{-0.5mm} i\al  \hs{-0.5mm} \int_0^t \Del u_K(t')\,dt' \hs{-0.5mm} + \hs{-0.5mm} i  \hs{-0.5mm} \int_0^tJ_{\le K}((\na \cd w_{K}(t')) v_{K}(t'))\,dt',\\
v_{K}(t)& \hs{-0.5mm} = \hs{-0.5mm} J_{\le K}v_0 \hs{-0.5mm} + \hs{-0.5mm} i\be  \hs{-0.5mm} \int_0^t \Del v_K(t')\,dt' \hs{-0.5mm} + \hs{-0.5mm} i \hs{-0.5mm} \int_0^t J_{\le K}((\ove{\na \cd w_{K}(t')}) u_{K}(t'))\,dt', \\
w_{K}(t)& \hs{-0.5mm} = \hs{-0.5mm} J_{\le K}w_0 \hs{-0.5mm} + \hs{-0.5mm} i\ga  \hs{-0.5mm} \int_0^t \Del w_K(t')\,dt' \hs{-0.5mm} - \hs{-0.5mm} i \hs{-0.5mm} \int_0^t J_{\le K}\na (u_{K}(t')\cd \ove{v_{K}(t')})\,dt'
\end{aligned}\mi.
\end{equation}
in $\ch^{s'}$ for any $s'\in \R$.

Fitst, we will show that the existence time of the solution to \eqref{K-CC} is independent of the approximation parameter $K$.
To prove this, we apply the short-time Fourier restriction norm method.
It is known that the following linear estimate holds.
\begin{lem}\label{lin}
Let $s\in \R$, $T\in (0, 1]$, and $\si\in \R\sm \{0\}$.
Suppose that $u\in C_TH^s\cap C_T^1 H^{s-2}$ and $f\in C_TH^s$ satisfy $i\pa_tu + \si\Del u= f$.
Then, we have
\[
\|u\|_{F_{\si}^s(T)}
\les \|u\|_{E^s(T)}+\|f\|_{G_{\si}^s(T)}.
\]
\end{lem}
\begin{proof}
See Proposition 4.1 in \cite{GO} and Proposition 3.2 in \cite{IKT}.
\end{proof}

To apply the short-time Fourier restriction norm method, we use the following nonlinear and energy estimates.
\begin{lem}\label{nonlin1}
Let $d\in \N$, $\til s\ge s>\max\{\fr{1}{2}, \fr{d-1}{2}\}$, and $\al, \be, \ga\in \R\sm \{0\}$.
Then, there exists $\theta \in (0, 1]$ depending only on $d$ such that
\begin{gather*}
\|J_{\le K}((\na\cd w)v)\|_{G_{\al}^{\til s}(T)}
\les T^{\theta}(\|w\|_{F_{\ga}^{s}(T)}\|v\|_{F_{\be}^{\til s}(T)} + \|w\|_{F_{\ga}^{\til s}(T)}\|v\|_{F_{\be}^{s}(T)}), \\
\|J_{\le K}((\ove{\na\cd w})u)\|_{G_{\be}^{\til s}(T)}
\les T^{\theta}(\|w\|_{F_{\ga}^{s}(T)}\|u\|_{F_{\al}^{\til s}(T)} + \|w\|_{F_{\ga}^{\til s}(T)}\|u\|_{F_{\al}^s(T)}), \\
\|\na J_{\le K}(u\cd \ove{v})\|_{G_{\ga}^{\til s}(T)}
\les T^{\theta}(\|u\|_{F_{\al}^{s}(T)}\|v\|_{F_{\be}^{\til s}(T)} + \|u\|_{F_{\al}^{\til s}(T)}\|v\|_{F_{\be}^{s}(T)})
\end{gather*}
for $K\in 2^{\N}\cup\{\infty\}$, $T\in (0, 1]$, $u\in F_{\al}^{\til s}(T)$, $v\in F_{\be}^{\til s}(T)$, and $w\in F_{\ga}^{\til s}(T)$.
In particular, for $K\in 2^{\N}\cup\{\infty\}$, let
\begin{equation}\label{N_H}
N_{K}(u, v, w):=(
J_{\le K}((\na\cd w)v), 
J_{\le K}((\ove{\na\cd w})u), 
\na J_{\le K}(u\cd \ove{v})
),
\end{equation}
then we have
\[
\|N_{K}(u, v, w)\|_{\bfg^{\til s}(T)}
\les T^{\theta}\|(u, v, w)\|_{\bff^{s}(T)}\|(u, v, w)\|_{\bff^{\til s}(T)}.
\]
\end{lem}

\begin{lem}\label{energy1}
Let $d\in \N$, $\al, \be, \ga\in \R\sm\{0\}$ satisfy $\al - \ga =0$ and $\be + \ga \neq 0$, and $\til s\ge s> \fr{1}{2}(d+1)$.
Then, there exists $\theta=\theta(d)\in (0, 1]$ such that
\begin{align*}
\|\bfu_K\|_{\bfe^{\til s}(T)}^2 
&\les (1+\|\bfu_K(0)\|_{\ch^{s}}^2)\|\bfu_K(0)\|_{\ch^{\til s}}^2 \\
&\quad + T^{\theta}(\|\bfu_K\|_{\bff^{s}(T)} + \|\bfu_K\|_{\bff^{s}(T)}^3)\|\bfu_K\|_{\bff^{\til s}(T)}^2
\end{align*}
for any $T\in (0, 1]$, $K\in 2^{\N}\cup\{\infty\}$, and a solution $\bfu_K=(u_{K}, v_{K}, w_{K})$ to \eqref{K-CC} which belongs to $C_T\ch^{\til s}\cap \bfe^{\til s}(T)\cap \bff^{\til s}(T)$.
\end{lem}
\noindent
We prove Lemmas \ref{nonlin1} and \ref{energy1} in Sections 4 and 6, respectively.

From Lemmas \ref{lin}--\ref{energy1}, we can obtain an a priori estimate and the existence of the solution to \eqref{K-CC} with the uniform existence time with respect to $K\in 2^{\N}$.
\begin{prop}\label{apri1}
Let $d\in \N$, $\al, \be, \ga\in \R\sm\{0\}$ satisfy $\al - \ga =0$ and $\be + \ga \neq 0$, $s> \fr{1}{2}(d+1)$, and $s\le \til s\le 2s+10$.
Then, for any $R>0$, there exists $T\in (0, 1]$ as follows: 
If $\bfu_0=(u_0, v_0, w_0) \in \ch^{2s+10}(\R^d)$ satisfies $\|\bfu_0\|_{\ch^s}\le R$, then for any $K\in 2^{\N}$, there exists a solution $\bfu_K=(u_{K}, v_{K}, w_{K})\in \cap_{s'\in \R} C([0, T]; \ch^{s'}(\R^d))$ of \eqref{K-CC}.
Moreover, $\bfu_K$ satisfies
\[
\|\bfu_K\|_{\bfe^{\til s}(T)} + \|\bfu_K\|_{\bff^{\til s}(T)}
\les (1+R)\|\bfu_0\|_{\ch^{\til s}},
\]
where the implicit constants are independent of $K$.
\end{prop}
\noindent
To prove this proposition, we confirm the continuity of the $\bfe^s(T)$- and $\bfg^s(T)$- norms.
\begin{lem}\label{Xs_conti}
Let $s'>\fr{d}{2}-1$, $K\in 2^{\N}\cup\{\infty\}$, $T_0\in (0, 1]$, and $(u, v, w) \in C([0,T_0] ; \ch^{s'+5}(\mathbb{\R}^d))$.
For $0< T\le T_0$, we define 
\begin{equation}\label{Xs}
X_K^{s'}(T):= \|(u, v, w)\|_{\bfe^{s'}(T)} + \|N_{K}(u, v, w)\|_{\bfg^{s'}(T)}.
\end{equation}
Then, $X_K^{s'}$ is non-decreasing and continuous in $T \in(0,T_0]$.
Also, we have
\[
\lim _{T \searrow 0} X_K^{s'}(T)\les \|(u(0), v(0), w(0))\|_{\ch^{s'}}.
\]
\end{lem}
\begin{proof}
See Lemma 8.1 in \cite{GO} and Lemma 4.2 in \cite{IKT}.
\end{proof}

\begin{proof}[Proof of Prposition \ref{apri1}]
Let $R>0$, $K\in 2^{\N}$, and $\bfu_0=(u_0, v_0, w_0)\in \ch^{2s+10}$ satisfy $\|\bfu_0\|_{\ch^s}\le R$.
From Lemma \ref{wp_K-CC}, there exist $T_0>0$ and a solution $\bfu_K=(u_{K}, v_{K}, w_{K})\in \cap_{s''\in \R} C([0, T_0]; \ch^{s''})$ of \eqref{K-CC}.
We set $X_K^{\til s}(T')$ by \eqref{Xs} with $s'=\til s$.
From Lemma \ref{Xs_conti}, $X_K^{\til s}$ is continuous on $[0, T_0]$ and $\lim_{T'\searrow 0}X_K^{\til s}(T')\les \|\bfu_0\|_{\ch^{\til s}}$.

First, we prove the case of $\til s = s$.
Then, it holds from Lemma \ref{energy1} that there exists $\til \theta\in (0, 1]$ such that
\begin{equation}\label{pr_apri1_2}
\|\bfu_K\|_{\bfe^{s}(T')}^2 
\les  (1+ \|\bfu_0\|_{\ch^{s}}^2)\|\bfu_0\|_{\ch^{s}}^2 + T'^{\til \theta}(\|\bfu_K\|_{\bff^{s}(T')} + \|\bfu_K\|_{\bff^{s}(T')}^3)\|\bfu_K\|_{\bff^{s}(T')}^2
\end{equation}
for $T'\in (0, T_0]$.
Thus, from \eqref{pr_apri1_2} and Lemmas \ref{lin} and \ref{nonlin1}, there exist $\theta\in (0, 1]$ and $C_2>0$ such that
\begin{equation}\label{pr_apri1_3}
X_K^{s}(T')
\le C_2( (1+ R)\|\bfu_0\|_{\ch^s} + T'^{\theta}(X_K^{s}(T')^{\fr{1}{2}} + X_K^{s}(T')^{\fr{3}{2}}))X_K^{s}(T')
\end{equation}
for $T'\in (0, T_0]$.
Since $X_K^{s}$ is continuous and satisfies $\lim_{T'\searrow 0}X_K^{s}(T')\le C\|\bfu_0\|_{\ch^s}$ for some $C>0$, by taking $C_2\ge C$ if necessary, a continuity argument yields $X_K^{s}(T')\le 3C_2 (1+R)\|\bfu_0\|_{\ch^s}=: A$ for
\[
T'\in (0, \min(T_0, (3C_2(A^{\fr{1}{2}}+A^{\fr{3}{2}}))^{-\fr{1}{\theta}})].
\]
When $T_0\le (3C_2(A^{\fr{1}{2}}+A^{\fr{3}{2}}))^{-\fr{1}{\theta}}=:T$, we have
\begin{equation}\label{pr_apri1_4}
\|\bfu_K(T_0)\|_{\ch^s} 
\le \til C \|\bfu_K\|_{\bfe^{s}(T_0)} 
\le 2\til CC_2 \|\bfu_0\|_{\ch^s}. 
\end{equation}
for some $\til C> 0$.
Thus,  from Lemma \ref{wp_K-CC}, there exist $T_1>0$ and a solution $\bfu_K\in \cap_{s'\in \R} C([0, T_0 + T_1]; \ch^{s'})$ of \eqref{K-CC} and we obtain \eqref{pr_apri1_3} for $T'\in (0, T_0 + T_1]$.
Hence, a continuity argument yields $X_K^{s}(T')\le  3C_2\|\bfu_0\|_{\ch^s}$ for $T'\in (0, \min(T_0 + T_1, T)]$, and thus we obtain \eqref{pr_apri1_4} with $T_0$ changed to $T_0 + T_1$ when $T_0+T_1\le T$.
Let $k\in \N$ satisfy $T_0 + kT_1\ge T$.
Repeating the above argument for $k$ times, we finish the proof.
We note that Lemma \ref{lin} yields
\begin{equation}\label{pr_apri1_5}
\|\bfu_K\|_{\bff^{s}(T)}
\les X_K^s(T)
\les (1+R)\|\bfu_0\|_{\ch^s}.
\end{equation}

Next, we consider the case of $s< \til s\le 2s +10$.
Lemmas \ref{lin}--\ref{energy1} and \eqref{pr_apri1_5} yield
\[
X_K^{\til s}(T')
\les (1+R)\|\bfu_0\|_{\ch^{\til s}} + T'^{\theta}(R+R^3)X_K^{\til s}(T')
\]
for $T'\in (0, T]$.
Since $X_K^{\til s}(T')<\infty$, we obtain $X_K^{\til s}(T')\les (1+R)\|\bfu_0\|_{\ch^{\til s}}$ if $T'\le T$ and $T'\les (R+R^3)^{-\fr{1}{\theta}}$.
This yields the desired bound.
\end{proof}

Next, we consider the estimates of the difference of the solutions to \eqref{K-CC} in $\ch^0$.
The nonlinear estimate in $\ch^0$ is as follows:
\begin{lem}\label{nonlin2}
Let $d\in\N$, $\al, \be, \ga\in \R\sm \{0\}$, and $s>\max\{1, \fr{d-1}{2}\}$.
Then, there exists $\theta=\theta(d) \in (0, 1]$ such that
\begin{gather*}
\|J_{\le K}((\na\cd w)v)\|_{G_{\al}^{0}(T)}
\les T^{\theta}\min \{\|w\|_{F_{\ga}^s(T)}\|v\|_{F_{\be}^{0}(T)}, \|w\|_{F_{\ga}^{0}(T)}\|v\|_{F_{\be}^s(T)}\}, \\
\|J_{\le K}(( \ove{\na\cd w})u)\|_{G_{\be}^{0}(T)}
\les T^{\theta}\min\{\|w\|_{F_{\ga}^s(T)}\|u\|_{F_{\al}^{0}(T)}, \|w\|_{F_{\ga}^{0}(T)}\|u\|_{F_{\al}^s(T)}\}, \\
\|\na J_{\le K}(u\cd \ove{v})\|_{G_{\ga}^{0}(T)}
\les T^{\theta}\min\{\|u\|_{F_{\al}^s(T)}\|v\|_{F_{\be}^0(T)}, \|u\|_{F_{\al}^0(T)}\|v\|_{F_{\be}^{s}(T)}\}
\end{gather*}
for $K\in 2^{\N}\cup\{\infty\}$, $T\in (0, 1]$, $u\in F_{\al}^{s}(T)$, $v\in F_{\be}^{s}(T)$, and $w\in F_{\ga}^{s}(T)$.
\end{lem}
\noindent
The above lemma is proved in Section 4.
We also need the following estimate in the case of $H_1\neq H_2$.
This is immediately shown by Remark \ref{rem_del} and Lemma \ref{nonlin1}.
Thus, we omit the proof.
\begin{lem}\label{nonlin3}
Let $d\in\N$, $\al, \be, \ga\in \R\sm \{0\}$, and $s>\max\{\fr{1}{2}, \fr{d-1}{2}\}$.
Then, there exists $\theta=\theta(d) \in (0, 1]$ such that
\begin{gather*}
\|J_{(K_1, K_2]}((\na \cd w_1) v_1)\|_{G_{\al}^{0}(T)}
\les K_1^{-s}T^{\theta}\|v_1\|_{F_{\be}^{s}(T)}\|w_1\|_{F_{\ga}^{s}(T)}, \\
\|J_{(K_1, K_2]}((\ove{\na \cd w_1}) u_1)\|_{G_{\be}^{0}(T)}
\les K_1^{-s}T^{\theta}\|u_1\|_{F_{\al}^{s}(T)}\|w_1\|_{F_{\ga}^{s}(T)}, \\
\|\na J_{(K_1, K_2]}(u_1\cd \ove{v_1}))\|_{G_{\ga}^{0}(T)}
\les K_1^{-s}T^{\theta}\|u_1\|_{F_{\al}^{s}(T)}\|v_1\|_{F_{\be}^{s}(T)}
\end{gather*}
for $K_1\le K_2\in 2^{\N}$, $T\in (0, 1]$, $u\in F_{\al}^{s}(T)$, $v\in F_{\be}^{s}(T)$, and $w\in F_{\ga}^{s}(T)$.
\end{lem}
\noindent
Next, we consider the energy estimate for the difference of the solutions to \eqref{K-CC} and \eqref{CC}.
\begin{lem}\label{dene1}
Let $d\in \N$, $\al, \be, \ga\in \R\sm\{0\}$ satisfy $\al - \ga =0$ and $\be + \ga \neq 0$, and $s> \fr{1}{2}(d+1)$.
Then, there exists $\theta=\theta(d)\in (0, 1]$ such that
\begin{align*}
&\|\bfu_1-\bfu_2\|_{\bfe^{0}(T)}^2 \nt\\
&\les (1+\|\bfu_1(0)\|_{\ch^{s}}^2 + \|\bfu_2(0)\|_{\ch^{s}}^2)\|\bfu_1(0)-\bfu_2(0)\|_{\ch^0}^2 \\
& + T^{\theta}(\|\bfu_1\|_{\bff^{s}(T)} + \|\bfu_2\|_{\bff^{s}(T)} + \|\bfu_1\|_{\bff^{s}(T)}^3 + \|\bfu_2\|_{\bff^{s}(T)}^3)\|\bfu_1-\bfu_2\|_{\bff^{0}(T)}^2  \\
& + K_1^{-(s-\fr{1}{2}(d+1))}T^{\theta}\|\bfu_1\|_{\bff^{s}(T)}^2(1+ \|\bfu_1\|_{\bff^{s}(T)}^2 + \|\bfu_2\|_{\bff^{s}(T)}^2)\|\bfu_1-\bfu_2\|_{\bff^{0}(T)}
\end{align*}
for $T\in (0, 1]$, $j=1,2$ and $\bfu_j=(u_j, v_j, w_j)\in C_T\ch^s\cap \bfe^{s}(T)\cap \bff^s(T)$ which is a solution to \eqref{K-CC} with $K = K_j\in2^{\N}\cup\{\infty\}$ and $K_1\le K_2$.
\end{lem}
\noindent
We prove Lemma \ref{dene1} in Section 7.

From Lemmas \ref{lin}, \ref{nonlin2}--\ref{dene1} and Proposition \ref{apri1}, we obtain the following estimate of the solutions to \eqref{K-CC} in $\ch^0$.
\begin{prop}\label{dest1}
Let $d\in \N$, $\al, \be, \ga\in \R\sm\{0\}$ satisfy $\al - \ga =0$ and $\be + \ga \neq 0$, and $s> \fr{1}{2}(d+1)$.
Then, for any $R>0$, there exists $T\in (0, 1]$ as follows: 
For $j=1, 2$, if $\bfu_{0, j}=(u_{0, j}, v_{0, j}, w_{0, j})\in \ch^{2s+10}$ satisfy $\|\bfu_{0, j}\|_{\ch^s}\le R$, then for any $K_j\in 2^{\N}$, there exists a solution $\bfu_j\in \cap_{s'\in \R} C([0, T]; \ch^{s'}(\R^d))$ of \eqref{K-CC} with $(u_0, v_0, w_0)=\bfu_{0, j}$ and $K=K_j$.
Moreover, $\bfu_1, \bfu_2$ satisfies
\begin{align*}
&\|\bfu_1-\bfu_2\|_{\bfe^0(T)} + \|\bfu_1-\bfu_2\|_{\bff^0(T)} \\
&\les (1+R^2)(\|\bfu_1(0)-\bfu_2(0)\|_{\ch^{0}} + \max\{K_1^{-(s-\fr{1}{2}(d+1))}, K_2^{-(s-\fr{1}{2}(d+1))}\}),
\end{align*}
where the implicit constant is independent of $K_1, K_2$.
\end{prop}
\begin{proof}
For $R>0$, let $T\in (0, 1]$ satisfy the condition of Proposition \ref{apri1}.
Also, we put $\vep=s-\fr{1}{2}(d+1)$.
We may assume $K_1\le K_2$.
We set $\mb{\til u}=(\til u, \til v, \til w):= \bfu_1 - \bfu_2$.
Lemma \ref{lin} yields
\begin{equation}\label{pr_dest1_1}
\|\mb{\til u}\|_{\bff^0(T')}
\les \|\mb{\til u}\|_{\bfe^0(T')} + \|N_{K_1}(u_1, v_1, w_1)- N_{K_2}(u_2, v_2, w_2)\|_{\bfg^0(T')}
\end{equation}
for $T'\in (0, T]$.
Also, from \eqref{K-CC}, we obtain
\begin{equation}\label{K-CC_diff}
\hi\{\begin{aligned}
\pa_t \til u &= i\al \Del \til u -iJ_{\le K_2}((\na \cd \til w) \til v) + iJ_{\le K_2}((\na \cd \til w) v_1)  \\
&\hs{20mm}+ iJ_{\le K_2}((\na \cd w_1) \til v) -i J_{(K_1, K_2]}((\na \cd w_1) v_1), \\
\pa_t \til v &= i\be \Del \til v -iJ_{\le K_2}((\ove{\na \cd \til w}) \til u) + iJ_{\le K_2}((\ove{\na \cd \til w}) u_1)  \\
&\hs{20mm}+ iJ_{\le K_2}((\ove{\na \cd w_1}) \til u) -i J_{(K_1, K_2]}((\ove{\na \cd w_1}) u_1), \\
\pa_t \til w &= i\ga \Del \til w + i\na J_{\le K_2}(\til u \cd \ove{\til v}) -i\na J_{\le K_2}(\til u \cd \ove{v_1}) \\
&\hs{20mm}-i\na J_{\le K_2}(u_1\cd \ove{ \til v}) +i\na J_{(K_1, K_2]}(u_1\cd \ove{v_1}),
\end{aligned}\mi.
\end{equation}
where we used  $J_{\le K_2}\til \bfu = \til \bfu$ and $J_{\le K_1}\bfu_1 = \bfu_1$.
Thus, it follows from Lemmas \ref{nonlin2} and \ref{nonlin3} and Proposition \ref{apri1} that 
\begin{align}
&\|N_{K_1}(u_1, v_1, w_1)- N_{K_2}(u_2, v_2, w_2)\|_{\bfg^0(T)} \label{pr_dest1_2}\\
&\les T^{\theta_1}(\|\bfu_1\|_{\bff^s(T')} + \|\bfu_2\|_{\bff^s(T)})\|\mb{\til u}\|_{\bff^0(T)}+ K_1^{-\vep}T^{\theta_1}\|\bfu_1\|_{\bff^s(T)}^2 \nt\\
&\les (R+R^2)T^{\theta_1}(\|\mb{\til u}\|_{\bff^0(T')} + K_1^{-\vep}) \nt
\end{align}
for $T'\in(0, T]$ and some ${\theta_1}={\theta_1}(d)\in (0, 1]$.
Also, Lemma \ref{dene1}, Proposition \ref{apri1}, and the Cauchy-Schwarz inequality yield
\begin{align}
&\|\mb{\til u}\|_{\bfe^0(T')}^2 \label{pr_dest1_3}\\
&\les (1+R^4)\|\mb{\til u}(0)\|_{\ch^0}^2 + T'^{\theta_2}(R+R^6)\|\mb{\til u}\|_{\bff^0(T')}^2 + T'^{\theta_2}(R^2+R^8)K_1^{-\vep}\|\mb{\til u}\|_{\bff^0(T')} \nt\\
&\les (1+R^4)\|\mb{\til u}(0)\|_{\ch^0}^2 + T'^{\theta_2}(R+R^6)\|\mb{\til u}\|_{\bff^0(T')}^2 + T'^{\theta_2}(R^2+R^{10})K_1^{-2\vep} \nt
\end{align}
for $T'\in(0, T]$ and some ${\theta_2}={\theta_2}(d)\in (0, 1]$.
From  \eqref{pr_dest1_1}--\eqref{pr_dest1_3}, by setting $Y(T'):= \|\til \bfu\|_{\bfe^0(T')} + \|N_{K_1}(u_1, v_1, w_1)- N_{K_2}(u_2, v_2, w_2)\|_{\bfg^0(T')}$, we have
\[
Y(T')
\les(1+R^2)\|\mb{\til u}(0)\|_{\ch^0}+ T'^{\theta}(R^{\fr{1}{2}}+R^3)Y(T') + T'^{\theta}(R+R^{5})K_1^{-\vep}
\]
for $T'\in(0, T]$ and some $\theta=\theta(d)\in (0, 1]$.
Since $Y(T)<\infty$, by taking $T\les (R^{\fr{1}{2}}+R^3)^{-\fr{1}{\theta}}$ if necessary, we have
\[
Y(T')
\les (1+R^2)(\|\mb{\til u}(0)\|_{\ch^0} + K_1^{-\vep})
\]
for $T'\in (0, T]$.
Therefore, Lemma \ref{lin} yields the desired bound.
\end{proof}

Now, we prove the existence of a solution to \eqref{CC} for smooth initial data.
\begin{prop}\label{sm_exist}
Let $d\in \N$, $\al, \be, \ga\in \R\sm\{0\}$ satisfy $\al-\ga=0$ and $\be + \ga \neq 0$, and 
$s> \fr{1}{2}(d+1)$.
Then, for any $R>0$, there exists $T\in (0, 1]$ such that for any $\bfu_0=(u_0, v_0, w_0)\in \ch^{2s+10}(\R^d)$ satisfying $\|\bfu_0\|_{\ch^s}\le R$, there exists a solution $\bfu=(u, v, w) \in C([0, T]; \ch^{2s+5}(\R^d))$ of \eqref{CC}.
\end{prop}
\begin{rem}\rm
If we perform the similar argument as the standard energy method, we can prove that $\bfu\in C([0, T]; \ch^{2s+10})$.
However, we do not show this statement since Proposition \ref{sm_exist} is sufficient to prove Theorem \ref{thm1}.
\end{rem}
\noindent
In the following, we denote $\langle\cd, \cd\rangle_{\ch^s}$ by the standard inner product in $\ch^s$.
\begin{proof}
For $R>0$, let $T\in (0, 1]$ satisfy the condition of Propositions \ref{apri1} and \ref{dest1}.
For $K\in 2^{\N}$, let $\bfu_{K}=(u_K, v_K, w_K)\in \cap_{s'\in \R}C([0, T]; \ch^{s'}(\R^d))$ be a solution to \eqref{K-CC} with initial data $\bfu_{K}(0)=J_{\le K}\bfu_0$.
Proposition \ref{dest1}  yields that $\{\bfu_K\}$ is a Cauchy sequence in $C([0, T]; \ch^{0})$.
In particular, there exists $\bfu=(u, v, w)\in C([0, T]; \ch^{0})$ such that $\lim_{K\to \infty}\|\bfu_{K}-\bfu\|_{C([0, T]; \ch^{0})}=0$.

Also, $\{\bfu_{K}(t)\}$ is a weak Cauchy sequence in $\ch^{2s+10}$ uniformly with respect to $t\in [0, T]$.
Indeed, for any $\phi\in \ch^{2s+10}$ and $a>0$, by choosing $\phi_a\in \ch^{2(2s+10)}$ such that $\|\phi-\phi_a\|_{\ch^{2s+10}}\le a$, the triangle inequality and Proposition \ref{apri1} yield
\begin{align*}
&|\langle \bfu_{K_1}(t) -  \bfu_{K_2}(t), \phi\rangle_{\ch^{2s+10}}| \\
&\le |\langle \bfu_{K_1}(t) -  \bfu_{K_2}(t), \phi-\phi_a\rangle_{\ch^{2s+10}}| + |\langle \bfu_{K_1}(t) -  \bfu_{K_2}(t), \phi_a\rangle_{\ch^{2s+10}}| \\
&\les (R+R^2)\|\phi-\phi_a\|_{\ch^{2s+10}} + \|\bfu_{K_1}(t) -  \bfu_{K_2}(t)\|_{\ch^0}\|\phi_a\|_{\ch^{2(2s+10)}}.
\end{align*}
Therefore, we have
\[
\limsup_{K_1, K_2\to\infty}\sup_{t\in [0, T]}|\langle \bfu_{K_1}(t) -  \bfu_{K_2}(t), \phi\rangle_{\ch^{2s+10}}|
\les  (R+R^2)a.
\]
Since $a>0$ is taken arbitrarily, we obtain 
\[
\lim_{K_1, K_2\to\infty}\sup_{t\in [0, T]}|\langle \bfu_{K_1}(t) -  \bfu_{K_2}(t), \phi\rangle_{\ch^{2s+10}}|=0.
\]
From the weak completeness of $\ch^{2s+10}$, there exists $\mb{u'}\in C_w([0, T]; \ch^{2s+10})$ which is a weak limit of $\{\bfu_K\}$, and then we have $\bfu=\mb{u'}$.

From Proposition \ref{apri1} and the triangle inequality 
\[
|\langle\bfu(t), \phi\rangle_{\ch^{2s+10}}|\le |\langle\bfu(t)-\bfu_K(t), \phi\rangle_{\ch^{2s+10}}| + |\langle\bfu_K(t), \phi\rangle_{\ch^{2s+10}}|
\]
for $\phi\in \ch^{2s+10}$, we have $\sup_{t\in [0, T]}\|\bfu(t)\|_{\ch^{2s+10}}<\infty$.
Thus, we obtain $\bfu\in C([0, T]; \ch^r)$ and $\lim_{K\to \infty}\|\bfu_K-\bfu\|_{C([0, T]; \ch^{r})}=0$ for $r< {2s+10}$.
Therefore, taking $r=2s+5$ and $K\to\infty$ on \eqref{K-CC Hs-sol} in $\ch^{2s+5-2}$, $\bfu$ is a $\ch^{2s+5}$-solution to \eqref{CC}.
\end{proof}

For the solutions to \eqref{CC}, the similar estimates as Propositions \ref{apri1} and \ref{dest1} hold.
We omit the proof.
\begin{prop}\label{apri2}
Let $d\in \N$, $\al, \be, \ga\in \R\sm\{0\}$ satisfy $\al - \ga =0$ and $\be + \ga \neq 0$, $s> \fr{1}{2}(d+1)$, and $s\le \til s\le 2s$.
Then, for any $R>0$ there exists $T\in (0, 1]$ as follows: 
If $\bfu_0=(u_0, v_0, w_0) \in \ch^{2s+10}(\R^d)$ satisfies $\|\bfu_0\|_{\ch^s}\le R$, then there exists a solution $\bfu=(u, v, w)\in C([0, T]; \ch^{2s+5}(\R^d))$ of \eqref{CC}.
Moreover, $\bfu$ satisfies
\[
\|\bfu\|_{\bfe^{\til s}(T)} + \|\bfu\|_{\bff^{\til s}(T)}
\les (1+R)\|\bfu_0\|_{\ch^{\til s}}.
\]
\end{prop}
\begin{prop}\label{dest2}
Let $d\in \N$, $\al, \be, \ga\in \R\sm\{0\}$ satisfy $\al - \ga =0$ and $\be + \ga \neq 0$, and $s> \fr{1}{2}(d+1)$.
Then, for any $R>0$, there exists $T\in (0, 1]$ as follows: 
For $j=1, 2$, if $\bfu_{0, j}=(u_{0, j}, v_{0, j}, w_{0, j}) \in \ch^{2s+10}(\R^d)$ satisfies $\|\bfu_{0, j}\|_{\ch^s}\le R$, then there exists a solution $\bfu_j=(u_j, v_j, w_j)\in C([0, T]; \ch^{2s+5}(\R^d))$ of \eqref{CC} with $(u_0, v_0, w_0)=\bfu_{0, j}$.
Moreover, $\bfu_1, \bfu_2$ satisfy 
\begin{equation}\label{dest2_1}
\|\bfu_1-\bfu_2\|_{\bfe^0(T)} + \|\bfu_1-\bfu_2\|_{\bff^0(T)}
\les (1+R^2)\|\bfu_{0, 1}-\bfu_{0, 2}(0)\|_{\ch^{0}}.
\end{equation}
\end{prop}
In Propositions \ref{apri2} and \ref{dest2}, we restrict the condition of $\til s$ in order to guarantee the continuity of $X_\infty^{\til s}$ which is defined by \eqref{Xs}.
However, this condition of $\til s$ is sufficient to prove the well-posedness for \eqref{CC}.

All that is left is to obtain the estimate of the difference of the solution to \eqref{CC} in $\ch^s$.
\begin{lem}\label{dene2}
Let $d\in \N$, $\al, \be, \ga\in \R\sm\{0\}$ satisfy $\al - \ga =0$ and $\be + \ga \neq 0$, and $s> \fr{1}{2}(d+1)$.
Then, there exists $\theta=\theta(d)\in (0, 1]$ such that
\begin{align}
&\|\bfu_1-\bfu_2\|_{\bfe^{s}(T)}^2 \nt\\
&\les (1 + \|\bfu_1(0)\|_{\ch^{s}}^2 + \|\bfu_2(0)\|_{\ch^{s}}^2 )\|\bfu_1(0)-\bfu_2(0)\|_{\ch^s}^2 \nt\\
&\quad + T^{\theta}(\|\bfu_1\|_{\bff^{s}(T)} + \|\bfu_2\|_{\bff^{s}(T)} + \|\bfu_1\|_{\bff^{s}(T)}^3 + \|\bfu_2\|_{\bff^{s}(T)}^3)\|\bfu_1-\bfu_2\|_{\bff^{s}(T)}^2  \nt\\
&\quad + T^{\theta}(\|\bfu_1\|_{\bff^{s}(T)} + \|\bfu_2\|_{\bff^{s}(T)})\|\bfu_1\|_{\bff^{2s}(T)}\|\bfu_1-\bfu_2\|_{\bff^{0}(T)}  \nt
\end{align}
for $T\in (0, 1]$, $j=1,2$ and $\bfu_j=(u_j, v_j, w_j)\in C_T\ch^{2s+5}$ which is a solution to \eqref{CC}.
\end{lem}
\noindent
Lemma \ref{dene2} is shown in Section 7.
From Lemma \ref{dene2}, we obtain the following estimate of the difference of the solutions to \eqref{CC} in $\ch^s$.
\begin{prop}\label{dest3}
Let $d\in \N$, $\al, \be, \ga\in \R\sm\{0\}$ satisfy $\al - \ga =0$ and $\be + \ga \neq 0$, and $s> \fr{1}{2}(d+1)$.
Then, for any $R>0$, there exists $T\in (0, 1]$ as follows: 
For $j=1, 2$, if $\bfu_{0, j}=(u_{0, j}, v_{0, j}, w_{0, j}) \in \ch^{2s+10}(\R^d)$ satisfies $\|\bfu_{0, j}\|_{\ch^s}\le R$, then there exists a solution $\bfu_j=(u_j, v_j, w_j)\in C([0, T]; \ch^{2s+5}(\R^d))$ of \eqref{CC} with $(u_0, v_0, w_0)=\bfu_{0, j}$.
Moreover, $\bfu_1, \bfu_2$ satisfy 
\begin{align}
&\|\bfu_1 - \bfu_2\|_{\bfe^{s}(T)} + \|\bfu_1 - \bfu_2\|_{\bff^{s}(T)} \nt\\
&\les (1+R^2)\|\bfu_1(0) - \bfu_2(0)\|_{\ch^{s}} + \|\bfu_1(0) - \bfu_2(0)\|_{\ch^0}^{1/2}\|\bfu_1(0)\|_{\ch^{2s}}^{1/2}. \nt
\end{align}
\end{prop}
\begin{proof}
For $R>0$, let $T\in (0, 1]$ satisfy the conditions of Propositions \ref{apri2} and \ref{dest2}.
Let $\mb{\til u}:= \bfu_1 - \bfu_2$.
Then, Lemma \ref{lin} yields
\begin{equation}\label{pr_dest3_1}
\|\mb{\til u}\|_{\bff^s(T')}
\les \|\mb{\til u}\|_{\bfe^s(T')} + \|N_{\infty}(u_1, v_1, w_1)- N_{\infty}(u_2, v_2, w_2)\|_{\bfg^s(T')}
\end{equation}
for $T'\in (0, T]$, where $N_{\infty}$ is defined by \eqref{N_H}.
Also, it follows from \eqref{K-CC_diff}, Lemma \ref{nonlin1}, and Proposition \ref{apri2} that 
\begin{equation}\label{pr_dest3_2}
\|N_{\infty}(u_1, v_1, w_1)- N_{\infty}(u_2, v_2, w_2)\|_{\bfg^s(T')} 
\les T'^{\theta_1}(R+R^2)\|\mb{\til u}\|_{\bff^s(T')}
\end{equation}
for $T'\in (0, T]$ and some $\theta_1=\theta_1(d)\in (0, 1]$.
Moreover, Lemma \ref{dene2} and Proposition \ref{apri2} with $\til s=2s$ yield
\begin{align}
\|\mb{\til u}\|_{\bfe^s(T')}^2
&\les (1+R^4)\|\mb{\til u}(0)\|_{\ch^s}^2 + T'^{\theta_2}(R+R^6)\|\mb{\til u}\|_{\bff^s(T')}^2 \label{pr_dest3_3}\\
&\quad + T'^{\theta_2}(R+R^2)\|\bfu_1\|_{\bff^{2s}(T')}\|\mb{\til u}\|_{\bff^0(T')} \nt\\
&\les (1+R^4)\|\mb{\til u}(0)\|_{\ch^s}^2 + T'^{\theta_2}(R+R^6)\|\mb{\til u}\|_{\bff^s(T')}^2 \nt\\
&\quad + T'^{\theta_2}(R+R^5)\|\bfu_1(0)\|_{\ch^{2s}}\|\mb{\til u(0)}\|_{\ch^0} \nt
\end{align}
for $T'\in (0, T]$ and some $\theta_2=\theta_2(d)\in (0, 1]$.
From  \eqref{pr_dest3_1}--\eqref{pr_dest3_3}, by putting $Z(T'):= \|\til \bfu\|_{\bfe^s(T')} + \|N_{\infty}(u_1, v_1, w_1)- N_{\infty}(u_2, v_2, w_2)\|_{\bfg^s(T')}$, we have
\begin{align}
Z(T')
&\les (1+R^2)\|\mb{\til u}(0)\|_{\ch^s} +  T'^{\theta}(R^{\fr{1}{2}}+R^3)Z(T') \nt\\
&\quad + T'^{\theta}(R^{\fr{1}{2}}+R^{\fr{5}{2}})\|\bfu_1(0)\|_{\ch^{2s}}^{\fr{1}{2}}\|\mb{\til u(0)}\|_{\ch^0}^{\fr{1}{2}} \nt
\end{align}
for $T'\in (0, T]$ and some $\theta=\theta(d)\in (0, 1]$.
Since $Z(T)<\infty$, by taking $T\les (R^{\fr{1}{2}}+R^3)^{-\fr{1}{\theta}}$ if necessary, it holds
\begin{equation}\label{pr_dest3_4}
Z(T')
\les (1+R^2)\|\mb{\til u}(0)\|_{\ch^s} + \|\bfu_1(0)\|_{\ch^{2s}}^{\fr{1}{2}}\|\mb{\til u}(0)\|_{\ch^{0}}^{\fr{1}{2}}
\end{equation}
for $T'\in (0, T]$.
Hence, Lemma \ref{lin} leads to the desired estimate.
\end{proof}

Now, we prove Theorem \ref{thm1}.
\begin{proof}[Proof of Theorem \ref{thm1}]
For $R>0$, let $T\in (0, 1]$ satisfy the conditions of Propositions \ref{sm_exist}, \ref{apri2}, \ref{dest2} and \ref{dest3}.
Also, let $\bfu_0=(u_0, v_0, w_0)\in \ch^s$ satisfy $\|\bfu_0\|_{\ch^s}\le R$.

We first prove the existence of a solution.
From Proposition \ref{sm_exist}, for $j\in \N$, let $J:=2^j$ and  $\bfu_{j}=(u_{j}, v_{j}, w_{j})\in C([0, T]; \ch^{2s+5})$ be a solution to \eqref{CC} with $\bfu_{j}(0)=P_{\le J}\bfu_0\in \ch^{2s+10}$.
Then, Proposition \ref{dest3} yields
\begin{align}
&\|\bfu_{j_1} - \bfu_{j_2}\|_{\bfe^{s}(T)} \label{pr_main1_0}\\
&\les (1+R^2)\|\bfu_{j_1}(0) - \bfu_{j_2}(0)\|_{\ch^{s}} + \|\bfu_{j_1}(0) - \bfu_{j_2}(0)\|_{\ch^0}^{\fr{1}{2}}\|\bfu_{j_1}(0)\|_{\ch^{2s}}^{\fr{1}{2}} \nt\\
&\les (1+R^2)\|\bfu_{j_1}(0) - \bfu_{j_2}(0)\|_{\ch^{s}} + 2^{-sj_1}\|\bfu_{j_1}(0) - \bfu_{j_2}(0)\|_{\ch^s}^{\fr{1}{2}}\cd 2^{sj_1}\|\bfu_{j_1}(0)\|_{\ch^{s}}^{\fr{1}{2}} \nt\\
&\le (1+R^2)\|\bfu_{j_1}(0) - \bfu_{j_2}(0)\|_{\ch^{s}} + R^{\fr{1}{2}}\|\bfu_{j_1}(0) - \bfu_{j_2}(0)\|_{\ch^s}^{\fr{1}{2}} \nt
\end{align}
for $j_1\le j_2$.
Since $\{\bfu_{j}(0)\}$ is a Cauchy sequence in $\ch^{s}$, $\{\bfu_{j}\}$ becomes a Cauchy sequence in $\bfe^{s}(T)$, in particular, in $C([0, T]; \ch^s)$.
Let $\bfu:=\lim_{j\to\infty}\bfu_j$, where the limit is taken in $\bfe^s(T)$. 
Then $\bfu$ is a $\ch^s$-solution to \eqref{CC} with initial data $\bfu_0$.
We note that Proposition \ref{apri2} and taking the limit as $j\to\infty$ yield
\begin{equation}\label{pr_main1_1}
\sup_{t\in [0, T]}\|\bfu(t)\|_{\ch^s}
\le C\|\bfu_0\|_{\ch^{s}}.
\end{equation}

Next, we show the uniqueness in $C([0, T]; \ch^s)\cap \bfe^s(T)\cap \bff^s(T)$.
We assume that $\mb{v}\in C([0, T]; \ch^s)\cap \bfe^s(T)\cap \bff^s(T)$ is also a solution to \eqref{CC} with initial data $\bfu_0$.
Then, by setting $\til Y(T'):= \|\bfu - \mb{v}\|_{\bfe^0(T')} + \|N_{\infty}(\bfu)- N_{\infty}(\mb{v})\|_{\bfg^0(T')}$ for $T'\in (0, T]$, the same argument as the proof of Proposition \ref{dest1} yields
\[
\til Y(T')
\les T'^{\theta}\Bl(\til Y(T)^{\fr{1}{2}} + \til Y(T)^{\fr{3}{2}}\Br)\til Y(T')
\]
for $T'\in(0, T]$ and some $\theta=\theta(d)\in (0, 1]$.
Since $\til Y(T)<\infty$, by taking $T_0\les (\til Y(T)^{\fr{1}{2}} + \til Y(T)^{\fr{3}{2}})^{-\fr{1}{\theta}}$, we have $\bfu=\mb{v}$ on $[0, T_0]$.
In particular, we obtain $\bfu(T_0)=\mb{v}(T_0)$.
Thus, the same argument as above yields $\bfu=\mb{v}$ on $[0, 2T_0]$.
By letting $k$ satisfy $kT_0 \ge T$ and repeating the above argument for $k$ times, we obtain the uniqueness on $[0, T]$.

Finally, we prove the continuity of the flow map.
Let $\bfu_{0, n}\in \ch^s$ satisfy $\|\bfu_{0, n}\|_{\ch^s}\le R$ for $n\in \N$ and $\lim_{n\to\infty}\bfu_{0, n} = \bfu_0$ in $\ch^s$.
Also, let $\bfu_n\in C([0, T]; \ch^s)\cap \bfe^s(T)\cap \bff^s(T)$ be the solution to \eqref{CC} with initial data $\bfu_{0, n}$. 
For $j\in\N$, let $J:= 2^j$ and $\bfu^{(j)}, \bfu_n^{(j)}\in C([0, T]; \ch^{2s+5})$ be the solutions to \eqref{CC} with initial data $P_{\le J}\bfu_0$ and $P_{\le J}\bfu_{0, n}$, respectiely.
From the proof of the existence and the uniqueness, we note that
\[
\lim_{j\to \infty}(\|\bfu - \bfu^{(j)}\|_{\bfe^{s}(T)} + \|\bfu_{n}^{(j)} - \bfu_{n}\|_{\bfe^{s}(T)})=0
\]
for any $n\in \N$.
Proposition \ref{dest3} and the same argument as \eqref{pr_main1_0} yield
\begin{align*}
&\|\bfu^{(j)} - \bfu_{n}^{(j)}\|_{\bfe^{s}(T)} \nt\\
&\les  (1+R^2)\|\bfu^{(j)}(0) - \bfu_{n}^{(j)}(0)\|_{\ch^{s}} + R^{\fr{1}{2}}\|\bfu^{(j)}(0) - \bfu_{n}^{(j)}(0)\|_{\ch^s}^{1/2} \\
&\les  (1+R^2)\|\bfu(0) - \bfu_n(0) \|_{\ch^{s}} +  R^{\fr{1}{2}}\|\bfu(0) - \bfu_n(0) \|_{\ch^{s}}^{1/2} 
\end{align*}
for $j\in \N$.
In particular, we have
\[
\limsup_{j\to\infty}\|\bfu^{(j)} - \bfu_{n}^{(j)}\|_{\bfe^{s}(T)}
\les (1+R^2)\|\bfu(0) - \bfu_n(0) \|_{\ch^{s}} +  R^{\fr{1}{2}}\|\bfu(0) - \bfu_n(0) \|_{\ch^{s}}^{1/2}.
\]
Therefore, by taking the limit superior with respect to $j$ in the following triangle inequality
\[
\|\bfu - \bfu_{n}\|_{\bfe^{s}(T)}
\le \|\bfu - \bfu^{(j)}\|_{\bfe^{s}(T)} + \|\bfu^{(j)} - \bfu_{n}^{(j)}\|_{\bfe^{s}(T)} + \|\bfu_{n}^{(j)} - \bfu_{n}\|_{\bfe^{s}(T)},
\]
we obtain
\[
\|\bfu - \bfu_{n}\|_{\bfe^{s}(T)} 
\les (1+R^2)\|\bfu(0) - \bfu_n(0) \|_{\ch^{s}} +  R^{\fr{1}{2}}\|\bfu(0) - \bfu_n(0) \|_{\ch^{s}}^{1/2}.
\]
This yields the continuity of the flow map.
\end{proof}

\section{Nonlinear estimate}
\subsection{Strichartz and bilinear Strichartz estimates}
In this subsection, we collect the Strichartz and the bilinear Strichartz estimates.
In the following, we say $(q, p)$ are admissible pair if $p, q$ satisfy $2\le p, q \le \infty$, $\fr{2}{q}=\fr{d}{2}- \fr{d}{p}$, and $(d, q, p)\neq (2, 2, \infty)$.
Also, we denote $L_x^p:=L_x^p(\R^d)$ and $L_t^qL_x^p:=L_t^q(\R; L_x^p(\R^d))$.
The Strichartz estimate is as follows:
\begin{lem}[cf. \cite{GV}, \cite{KeTa}]
Let $d\in \N$, $\si\in \R\sm\{0\}$, and $(q, p)$ be an admissible pair.
Then, we have
\[
\|e^{it\si\Del}\varphi\|_{L_t^qL_x^p}
\les \|\varphi\|_{L_x^2}
\]
for any $\varphi\in L^2(\R^d)$.
\end{lem}
From the Strichartz estimate, we obtain the following corollary.
For the proof, see Lemma 2.3 in \cite{GTV}.
\begin{cor}\label{str_cor}
Let $\si \in \R\sm\{0\}$, $j\in \N_0$, and $(q, p)$ be an admissible pair.
Then, we have
\[
\|Q_j^{\si}u\|_{L_t^qL_x^p}
\les 2^{\fr{j}{2}}\|Q_j^{\si}u\|_{L_t^2L_x^2}
\]
for any $u \in L_t^2(\R ; L_x^2(\R^d))$.
In particular, the definition of $F_{N, \si, T}$ given by \eqref{Fnst} yields
\[
\|g(t)\cd u\|_{L_t^qL_x^p}
\les \|u\|_{F_{N, \si, T}}
\]
for any  $N\in 2^{\N_0}$, $T\in (0, 1]$, $t_N\in \R$, $u\in F_{N, \si, T}$, and $g\in L_t^\infty(\R) $ satisfying $|g(t)|\le \eta_0(NT^{-1}(t-t_N))$ for $t\in\R$, where $\eta_0$ is defined by \eqref{eta}.
Thus we have
\[
\|g(t)\cd u\|_{L_T^q L_x^p}\les \|u\|_{F_{N, \si}(T)}
\]
for $u\in F_{N, \si}(T)$, which yields $\|u\|_{L_T^\infty L_x^2}\les \|u\|_{F_{N, \si}(T)}$ for $u\in F_{N, \si}(T)$.
\end{cor}
\noindent
From Corollary \ref{str_cor}, we immediately obtain $F_{\si}^s(T)\hookrightarrow E^s(T)$.

The bilinear Strichartz estimate is as follows.
\begin{lem}[\cite{Hirayama2014}, Lemma 3.1]\label{bilin1}
Let $\si_1, \si_2 \in \R\sm\{0\}$ and $j_1, j_2\in \N_0$.
\\
(i) Let $d\in \N$ and  $N_1, N_2 \in 2^{\N_0}$ satisfy $N_1 \gg N_2$.
Then, we have
\begin{align}
&\|(P_{N_1}Q_{j_1}^{\si_1} f_1)(P_{N_2}Q_{j_2}^{\si_2} f_2)\|_{L_t^2L_x^2}  \nt\\
&\les N_2^{\fr{d-1}{2}}N_1^{-\fr{1}{2}}2^{\fr{j_1}{2}}2^{\fr{j_2}{2}}\|P_{N_1}Q_{j_1}^{\si_1} f_1\|_{L_t^2L_x^2}\|P_{N_2}Q_{j_2}^{\si_2} f_2\|_{L_t^2L_x^2} \nt
\end{align}
for $f_1, f_2\in L_t^2(\R ; L_x^2(\R^d))$.
\\
(ii) Let $H>0$ and $A\subset \R^2$ satisfy $A\subset\{(\xi_1, \xi_2)\in \R^2 \mid |\si_1\xi_1-\si_2\xi_2|\ge H\}$.
Then, we have
\begin{align}
&\Bl\|\iint_{\substack{\tau_1+\tau_2=\tau \\ \xi_1+\xi_2=\xi}}\mb{1}_{A}(\xi_1, \xi_2)\cf_{t, x}[Q_{j_1}^{\si_1}f_1](\tau_1, \xi_1)\cf_{t, x}[Q_{j_2}^{\si_2}f_2](\tau_2, \xi_2) \Br\|_{L_{\tau}^2(\R; L_{\xi}^2(\R))}  \nt\\
&\les H^{-\fr{1}{2}}2^{\fr{j_1}{2}}2^{\fr{j_2}{2}}\|P_{N_1}Q_{j_1}^{\si_1} f_1\|_{L_t^2L_x^2}\|P_{N_2}Q_{j_2}^{\si_2} f_2\|_{L_t^2L_x^2} \nt
\end{align}
for $f_1, f_2\in L_t^2(\R ; L_x^2(\R))$.
\end{lem}

From Lemma \ref{bilin1} and the definition of $F_{N, \si, T}$, the following corollary holds.

\begin{cor}\label{bilin2}
Let $\si_1, \si_2 \in \R\sm\{0\}$. 
\\
(i)
Let $d\in \N$ and $N_1, N_2 \in 2^{\N_0}$ satisfy $N_1 \gg N_2$.
Then, we have
\[
\|(g_1(t)\cd u_1)(g_2(t)\cd u_2)\|_{L_t^2L_x^2}
\les N_2^{\fr{d-1}{2}}N_1^{-\fr{1}{2}}\|u_1\|_{F_{N_1, \si_1, T}}\|u_2\|_{F_{N_2, \si_2, T}}
\]
for $T\in (0, 1]$, $t_{N_1}, t_{N_2}\in\R$, $u_1\in F_{N_1, \si_1, T}$ and $u_2\in F_{N_2, \si_2, T}$, and $g_1, g_2\in L_t^\infty(\R) $ satisfying $|g_j(t)|\le \eta_0(N_jT^{-1}(t-t_{N_j}))$ for $j=1, 2$ and $t\in\R$.
The same estimate holds if we change $F_{N_1, \si_1, T}$ and $F_{N_2, \si_2, T}$ into $F_{N_1, \si_1}(T)$ and $F_{N_2, \si_2}(T)$, respectively.
\\
(ii) Let $H>0$ and $A\subset \R^2$ satisfy $A\subset\{(\xi_1, \xi_2)\in \R^2 \mid |\si_1\xi_1-\si_2\xi_2|\ge H\}$.
Then, we have
\begin{align}
&\Bl\|\iint_{\substack{\tau_1+\tau_2=\tau \\ \xi_1+\xi_2=\xi}}\mb{1}_{A}(\xi_1, \xi_2)\cf_{t, x}[g_1(t)\cd u_1](\tau_1, \xi_1)\cf_{t, x}[g_2(t)\cd u_2](\tau_2, \xi_2) \Br\|_{L_{\tau}^2(\R; L_{\xi}^2(\R))}  \nt\\
&\les H^{-\fr{1}{2}}\|u_1\|_{F_{N_1, \si_1}(T)}\|u_2\|_{F_{N_2, \si_2}(T)} \nt
\end{align}
for $T\in (0, 1]$, $N_1, N_2\in 2^{\N_0}$, $t_{N_1}, t_{N_2}\in \R$, $u_1\in F_{N_1, \si_1}(T)$, $u_2\in F_{N_2, \si_2}(T)$, and $g_1, g_2\in L_t^\infty(\R) $ satisfying $|g_j(t)|\le \eta_0(N_jT^{-1}(t-t_{N_j}))$ for $j=1, 2$ and $t\in\R$.
\end{cor}

From Lemma \ref{bilin1}, we also obtain the following corollary, which is used to consider the High$\times$High$\to$Low interaction in the proof of nonlinear estimates.
\begin{cor}\label{bilin3}
Let $d\in \N$, $\si_1, \si_2 \in \R\sm\{0\}$, $a\in(0, 1)$, $j_1, j_2\in\N_0$, and $N_1, N_2 \in 2^{\N_0}$ satisfy $N_1 \gg N_2$, we have
\begin{align}
&\|(P_{N_1}Q_{j_1}^{\si_1} f_1)(P_{N_2}Q_{j_2}^{\si_2} f_2)\|_{L_t^2L_x^2} \nt\\
&\les N_2^{\fr{d-1}{2}(1-a)}N_1^{-\fr{1}{2}(1-a)+\fr{d}{2}a}2^{\fr{j_1}{2}}2^{\fr{j_2}{2}(1-a)} \|P_{N_1}Q_{j_1}^{\si_1} f_1\|_{L_t^2L_x^2}\|P_{N_2}Q_{j_2}^{\si_2} f_2\|_{L_t^2L_x^2} \nt
\end{align}
for $f_1, f_2\in L_t^2(\R ; L_x^2(\R^d))$.
In particular, we have
\begin{align}
&\|(g_1(t)\cd u_1)(P_{N_2}Q_{j}^{\si_2}u_2)\|_{L_t^2L_x^2} \nt\\
&\les N_2^{\fr{d-1}{2}(1-a)}N_1^{-\fr{1}{2}(1-a)+\fr{d}{2}a}2^{\fr{j}{2}(1-a)}\|u_1\|_{F_{N_1, \si_1, T}}\|P_{N_2}Q_{j}^{\si_2}u_2\|_{L_t^2L_x^2} \nt
\end{align}
for $T\in (0, 1]$,  $j\in\N_0$, $t_{N_1}\in\R$, $u_1\in F_{N_1, \si_1, T}$, $u_2\in F_{N_2, \si_2, T}$, and $g_1\in L_t^\infty(\R) $ satisfying $|g_1(t)|\le \eta_0(N_1T^{-1}(t-t_{N_1}))$ for $t\in\R$.
\end{cor}
\begin{proof}
This estimate is obtained by the interpolation of Lemma \ref{bilin1} and the following inequality 
\begin{align}
\|(P_{N_1}Q_{j_1}^{\si_1} f_1)(P_{N_2}Q_{j_2}^{\si_2} f_2)\|_{L_t^2L_x^2} 
&\les \|P_{N_1}Q_{j_1}^{\si_1} f_1\|_{L_t^\infty L_x^\infty}\|(P_{N_2}Q_{j_2}^{\si_2} f_2)\|_{L_t^2L_x^2} \nt\\
&\les N_1^{\fr{d}{2}}2^{\fr{j_1}{2}} \|P_{N_1}Q_{j_1}^{\si_1} f_1\|_{L_t^2L_x^2}\|(P_{N_2}Q_{j_2}^{\si_2} f_2)\|_{L_t^2L_x^2} \nt
\end{align}
which follows from Corollary \ref{str_cor} and the H\"older and Bernstein inequalities.
\end{proof}

\subsection{Proof of Lemmas \ref{nonlin1} and \ref{nonlin2}}
In this subsection, we prove Lemmas \ref{nonlin1} and \ref{nonlin2}.
In the following, for $k\in \N$, we write $N_1^*, \ldots, N_k^*$ to denote $N_1, \ldots, N_k\in 2^{\N_0}$ in order to satisfy $N_1^*\ge \ldots\ge N_k^*$.

\begin{proof}[Proof of Proposition \ref{nonlin1}]
We omit $J_{\le K}$ since it is harmless.
It suffices to show
\begin{align} \label{pr_nlin1_1}
\|P_{N_3} ((\na\cd P_{N_1}w)P_{N_2}v)\|_{G_{{N_3}, \al, T}} 
\les T^{\theta}C_{\mb{N}}\|P_{N_1}w\|_{F_{N_1, \ga, T}}\|P_{N_2}v\|_{F_{N_2, \be, T}}, 
\end{align}
\begin{align} \label{pr_nlin1_2}
\|P_{N_3}((\ove{\na\cd P_{N_1}w})P_{N_2}u)\|_{G_{{N_3}, \be, T}} 
\les  T^{\theta}C_{\mb{N}}\|P_{N_1}w\|_{F_{N_1, \ga, T}}\|P_{N_2}u\|_{F_{N_2, \al, T}}, 
\end{align}
\begin{align} \label{pr_nlin1_3}
\|\na P_{N_3}(P_{N_1}u\cd \ove{P_{N_2}v})\|_{G_{{N_3}, \ga, T}} 
\les  T^{\theta}C_{\mb{N}}\|P_{N_1}u\|_{F_{N_1, \al, T}}\|P_{N_2}v\|_{F_{N_2, \be, T}} 
\end{align}
for any $T\in (0, 1]$ and some $\theta=\theta(d)>0$ and $\del=\del(s, d)>0$ when $\til s\ge s>\max\{\fr{1}{2}, \fr{d-1}{2}\}$, where $C_{\mb{N}}$ is defined for $\mb{N}=(N_1, N_2, N_3)$ by
\[
C_{\mb{N}}
:=N_3^{-\til s}(N_1^{\til s}N_2^s + N_1^sN_2^{\til s})(N_3^*)^{-\del}.
\]
We note that the left-hand sides of \eqref{pr_nlin1_1}--\eqref{pr_nlin1_3} vanish if $N_1^*\gg N_2^*$, and thus we may assume $N_1^*\sim N_2^*$.
We only prove \eqref{pr_nlin1_2}, but \eqref{pr_nlin1_1} and \eqref{pr_nlin1_3} can be shown in the same manner.
For $T>0$ and $m\in\Z$, we define
\begin{equation}\label{g_m}
g_m(t):=\mb{1}_{[(m-1)T/N_1^*, mT/N_1^*]}(t).
\end{equation}
Then, we have
\begin{equation}\label{g_m1}
|g_m(t)|\le \eta_0(NT^{-1}(t-t_m))
\end{equation}
for $N=N_1, N_2, N_3$ and $t\in\R$, where $t_m=(m-1)T/N_1^*$.
Also, for $T>0$ and $t_{N_3}\in\R$, we set 
\[
\til\eta_{0}(t):=\eta_0(N_3T^{-1}(t-t_{N_3})).
\]
In the following, we set $L_t^qL_x^p:=L_t^q(\R ;L_x^p(\R^d))$.\\
\noindent
(I) The case of $N_1^*\sim N_3^*$ (High$\times$High$\to$High, $N_1^*/N_3\sim 1$).
\\
(I-a) The case of $d=1, 2$.

The H\"{o}lder inequality and Corollary \ref{str_cor} yield
\begin{align}
&\|\til\eta_{0}(t) P_{N_3}((\ove{\na\cd P_{N_1}w})P_{N_2}u)\|_{L_t^2L_x^2} \label{Ia1}\\
&\le \sum_{|m|\les N_1^*/N_3}\|\til\eta_{0}(t) P_{N_3}((\ove{g_{m}(t)\na\cd P_{N_1}w})g_{m}(t)P_{N_2}u)\|_{L_t^2L_x^2} \nt\\
&\les \sup_{m\in \Z}|N_3^{-1}T|^{\fr{1}{2}-\fr{d}{4}}\|g_{m}(t)\na\cd  P_{N_1} w\|_{L_t^{8/d}L_x^4}\|g_{m}(t)P_{N_2} u\|_{L_t^{8/d}L_x^4} \nt\\
&\les T^{\fr{1}{2}-\fr{d}{4}}(N_1^*)^{\fr{1}{2}+\fr{d}{4}}\|P_{N_1}w\|_{F_{N_1, \ga, T}}\|P_{N_2}u\|_{F_{N_2, \al, T}}. \nt
\end{align}
Also, from the definitions of $X_{N, \si}$ and $G_{N, \si, T}$, we have
\begin{equation}\label{estG}
\|f\|_{G_{N_3, \be, T}} \les T^{\fr{1}{2}}N_3^{-\fr{1}{2}}\sup_{t_{N_3}\in\R}\|\eta_0(N_3T^{-1}(t-t_{N_3}))\cd f\|_{L_t^2L_x^2}.
\end{equation}
Thus, \eqref{Ia1} and \eqref{estG} yield
\begin{equation}
\|P_{N_3} ( \hs{-0.5mm} (\ove{\na \hs{-0.5mm} \cd  \hs{-0.5mm} P_{N_1} \hs{-0.2mm} w})P_{N_2} \hs{-0.2mm} u)\|_{G_{{N_3}, \be, T}}
 \hs{-0.5mm} \les  \hs{-0.5mm} T^{1-\fr{d}{4}} \hs{-0.5mm} (N_1^*)^{\fr{d}{4}} \hs{-0.5mm} \|P_{N_1} \hs{-0.2mm} w\|_{F_{N_1, \ga, T}} \hs{-0.5mm} \|P_{N_2} \hs{-0.2mm} u\|_{F_{N_2, \al, T}}.
\label{Ia2}
\end{equation}
This yields \eqref{pr_nlin1_2} if $s>\fr{d}{4}$, $\theta=1-\fr{d}{4}$, and $\del=s-\fr{d}{4}$ since we have $ (N_1^*)^{\fr{d}{4}}\les N_3^{-\til s}(N_1^{\til s}N_2^s+N_1^sN_2^{\til s})(N_3^*)^{-(s-\fr{d}{4})}$.
\\
(I-b) The case of $d\ge 3$.

The Sobolev embedding $W^{\fr{d-2}{2}, \fr{2d}{d-1}}(\R^d)\hookrightarrow L^{2d}(\R^d)$ and the H\"{o}lder inequality yield
\begin{align}
&\sum_{|m|\les N_1^*/N_3}\|\til\eta_{0}(t) P_{N_3}((\ove{g_{m}(t)\na\cd P_{N_1}w})g_{m}(t)P_{N_2}u)\|_{L_t^2L_x^2}  \label{Ib1}\\
&\les \sup_{m \in \Z}\|g_{m}(t)\na\cd  P_{N_1} w\|_{L_t^4L_x^{2d/(d-1)}}\|g_{m}(t)P_{N_2} u\|_{L_t^4L_x^{2d}} \nt\\
&\les (N_3^*)^{\fr{d}{2}}\|P_{N_1}w\|_{F_{N_1, \ga, T}}\|P_{N_2}u\|_{F_{N_2, \al, T}}.\nt
\end{align}
Thus, we have from \eqref{estG} and \eqref{Ib1}  that
\begin{equation}
\|P_{N_3} ( \hs{-0.5mm} (\ove{\na \hs{-0.5mm} \cd  \hs{-0.5mm} P_{N_1} \hs{-0.2mm} w})P_{N_2} \hs{-0.2mm} u)\|_{G_{{N_3}, \be, T}}
 \hs{-0.5mm} \les  \hs{-0.5mm} T^{\fr{1}{2}} \hs{-0.5mm} (N_1^*)^{\fr{d-1}{2}} \hs{-0.5mm} \|P_{N_1} \hs{-0.2mm} w\|_{F_{N_1, \ga, T}} \hs{-0.5mm} \|P_{N_2} \hs{-0.2mm} u\|_{F_{N_2, \al, T}}, \label{Ib2}
\end{equation}
which yields \eqref{pr_nlin1_2} if $s>\fr{d-1}{2}$, $\theta=\fr{1}{2}$, and $\del=s-\fr{d-1}{2}$.
\\
(II) The case of $N_2^*\gg N_3^*$.

We divide this case into two cases. 
\\
(II-i) $N_3\neq N_3^*$ (High$\times$Low$\to$High, $N_1^*/N_3\sim 1$). 

We only consider the case of $N_2=N_3^*$.
The case of $N_1=N_3^*$ is similar.
Since $N_1\gg N_2$, Corollary \ref{bilin1} yields
\begin{align} \label{IIi1}
&\sum_{|m|\les N_1^*/N_3}\|\til\eta_{0}(t) P_{N_3}((\ove{g_{m}(t)\na\cd P_{N_1}w})g_{m}(t)P_{N_2}u)\|_{L_t^2L_x^2} \\
&\les N_2^{\fr{d-1}{2}}N_1^{-\fr{1}{2}}\cd N_1\|P_{N_1}w\|_{F_{N_1, \ga, T}}\|P_{N_2}u\|_{F_{N_2, \al, T}} \nt\\
&= (N_3^*)^{\fr{d-1}{2}}N_1^{\fr{1}{2}}\|P_{N_1}w\|_{F_{N_1, \ga, T}}\|P_{N_2}u\|_{F_{N_2, \al, T}}. \nt
\end{align}
Thus, it holds from \eqref{estG} and \eqref{IIi1} that
\begin{equation}
\|P_{N_3} ( \hs{-0.5mm} (\ove{\na \hs{-0.5mm} \cd  \hs{-0.5mm} P_{N_1} \hs{-0.2mm} w})P_{N_2} \hs{-0.2mm} u)\|_{G_{{N_3}, \be, T}}
 \hs{-0.5mm} \les  \hs{-0.5mm} T^{\fr{1}{2}} \hs{-0.5mm} (N_1^*)^{\fr{d-1}{2}} \hs{-0.5mm} \|P_{N_1} \hs{-0.2mm} w\|_{F_{N_1,  \hs{-0.2mm} \ga,  \hs{-0.2mm} T}} \hs{-0.5mm} \|P_{N_2} \hs{-0.2mm} u\|_{F_{N_2,  \hs{-0.2mm} \al,  \hs{-0.2mm} T}} \hs{-0.2mm} , \label{IIi2}
\end{equation}
which yields \eqref{pr_nlin1_2} if $s>\fr{d-1}{2}$, $\theta=\fr{1}{2}$, and $\del=s-\fr{d-1}{2}$.
\\
(II-ii) $N_3= N_3^*$ (High$\times$High$\to$Low, $N_1^*/N_3\gg 1$).

Let $0<a<1$ be chosen later.
Then, the H\"{o}lder inequality (with $(\fr{1}{2}, \fr{1}{\infty}) + (\fr{1}{\infty}, \fr{1}{2}) + (\fr{1}{2}, \fr{1}{2}) = (1, 1)$), $N_2\gg N_3$, and Corollary \ref{str_cor} and \ref{bilin3} yield
\begin{align*}
&2^{-\fr{j_3}{2}}\sum_{|m|\les N_1^*/N_3}\|\til\eta_{0}(t) P_{N_3}Q_{j_3}^{\be}((\ove{g_{m}(t)\na\cd P_{N_1}w})g_{m}(t)P_{N_2}u)\|_{L_t^2L_x^2} \\
&\les 2^{-\fr{j_3}{2}}\fr{N_1^*}{N_3^*} \sup_{\substack{m\in\Z \\ \|f_m\|_{L_{t, x}^2}=1}}\hs{-3mm}\Bl|\int_{\R}\int_{\R^d} (\ove{g_{m}(t)\na \hs{-0.5mm} \cd \hs{-0.5mm} P_{N_1}w})(g_{m}(t) P_{N_2}u \hs{-0.5mm} \cd \hs{-0.5mm} \ove{\til \eta_0(t) P_{N_3}Q_{j_3}^\be f_m})   dxdt\Br| \\
&\les 2^{-\fr{j_3}{2}}\fr{N_1^*}{N_3^*}\cd N_1\cd (N_1^*T^{-1})^{-\fr{1}{2}}\cd N_3^{\fr{d-1}{2}(1-a)}N_2^{-\fr{1}{2}(1-a)+\fr{d}{2}a}2^{\fr{j_3}{2}(1-a)} \\
&\hs{20mm} \times\|P_{N_1}w\|_{F_{N_1, \ga, T}}\|P_{N_2}u\|_{F_{N_2, \al, T}} \\
&\le T^{\fr{1}{2}}(N_1^*)^{1+\fr{d+1}{2}a}(N_3^*)^{\fr{d-3}{2}-\fr{d-1}{2}a}2^{-\fr{j_3}{2}a}\|P_{N_1}w\|_{F_{N_1, \ga, T}}\|P_{N_2}u\|_{F_{N_2, \al, T}}
\end{align*}
for $j_3\in \N_0$.
By the definition of $G_{N_3, \be, T}$ and summing up over $j_3\in\N_0$, we obtain
\begin{align}
&\|P_{N_3}((\ove{\na\cd P_{N_1}w})P_{N_2}u)\|_{G_{{N_3}, \be, T}} \label{IIii2}\\
&\les T^{\fr{1}{2}}(N_1^*)^{1+\fr{d+1}{2}a}(N_3^*)^{\fr{d-3}{2}-\fr{d-1}{2}a}\|P_{N_1}w\|_{F_{N_1, \ga, T}}\|P_{N_2}u\|_{F_{N_2, \al, T}}, \nt
\end{align}
which yields \eqref{pr_nlin1_2} if $s> \max\{\fr{d-1}{2}, \fr{1}{2}\}$, $\theta=\fr{1}{2}$, and $\del=\fr{1}{2}(s-\max\{\fr{d-1}{2}, \fr{1}{2}\})$.
Indeed, \eqref{IIii2} is shown by
\begin{equation}\label{pr_nlin1_4}
(N_1^*)^{1+\fr{d+1}{2}a}(N_3^*)^{\fr{d-3}{2}}
\les (N_3)^{-\til s}(N_1^{\til s}N_2^{s} + N_1^{s}N_2^{\til s})(N_3^*)^{-\fr{1}{2}(s-\max\{\fr{d-1}{2}, \fr{1}{2}\})}
\end{equation}
with $a=\min\{\fr{1}{2}, \fr{1}{d+1}(s-\max\{\fr{d-1}{2}, \fr{1}{2}\})\}$.
In the case of $d=1$, we can obtain \eqref{pr_nlin1_4} by dividing the condition of $\til s$ into the case of $\fr{1}{2}< \til s\le 1$ and the case of $\til s>1$.
Also, the case of $d\ge 2$ is straightforward because $\fr{d-3}{2}+\til s>0$.
\end{proof}

\begin{proof}[Proof of Lemma \ref{nonlin2}]
It suffices to show
\begin{equation*}
\|P_{N_3}((\na\cd P_{N_1}w)P_{N_2}v)\|_{G_{{N_3}, \al, T}}
\les  T^{\theta}C_{\mb{N}}'\|P_{N_1}w\|_{F_{N_1, \ga, T}}\|P_{N_2}v\|_{F_{N_2, \be, T}},
\end{equation*}
\begin{equation}\label{pr_nlin2_2}
\|P_{N_3}(( \ove{\na\cd P_{N_1}w})P_{N_2}u)\|_{G_{{N_3}, \be, T}}
\les  T^{\theta}C_{\mb{N}}'\|P_{N_1}w\|_{F_{N_1, \ga, T}}\|P_{N_2}u\|_{F_{N_2, \be, T}},
\end{equation}
\begin{equation*}
\|\na P_{N_3}(P_{N_1}u\cd \ove{P_{N_2}v})\|_{G_{{N_3}, \ga, T}} 
\les T^{\theta}C_{\mb{N}}'\|P_{N_1}u\|_{F_{N_1, \al, T}}\|P_{N_2}v\|_{F_{N_2, \be, T}}
\end{equation*}
for any $T\in (0, 1]$ and some $\theta=\theta(d)>0$ and $\del=\del(s, d)>0$ if $ s>\max\{\fr{d-1}{2}, 1\}$, where $C_{\mb{N}}'$ is defined for $\mb{N}=(N_1, N_2, N_3)$ by
\[
C_{\mb{N}}'
:=\min\{N_1^s, N_2^s\}(N_3^*)^{-\del}.
\]
We only prove \eqref{pr_nlin2_2}.
Other estimates are shown in the same manner.
We divide four cases as in the proof of Lemma \ref{nonlin1}.

In the case of (I-a), \eqref{Ia2} yields \eqref{pr_nlin2_2} if $s>\fr{d}{4}$, $\theta=1-\fr{d}{4}$, and $\del=s-\fr{d}{4}$.
In the case of (I-b) and (II-i), it holds from \eqref{Ib2} and \eqref{IIi2} that we have \eqref{pr_nlin2_2} if $s>\fr{d-1}{2}$, $\theta=\fr{1}{2}$, and $\del=s-\fr{d-1}{2}$.
In the case of (II-ii), \eqref{IIii2} yields \eqref{pr_nlin2_2} if $s> \max\{\fr{d-1}{2}, 1\}$, $\theta=\fr{1}{2}$, and $\del=\fr{1}{2}(s-\max\{\fr{d-1}{2}, 1\}$ since we have $(N_1^*)^{1+\fr{d+1}{2}a}(N_3^*)^{\fr{d-3}{2}} \les \min\{N_1^s, N_2^s\}(N_3^*)^{-\fr{1}{2}(s-\max\{\fr{d-1}{2}, 1\})}$ with $a=\min\{\fr{1}{2}, \fr{1}{d+1}(s-\max\{\fr{d-1}{2}, 1\})\}$.
\end{proof}

\section{Multilinar estimates}
\subsection{Trilinear estimates}
Also in this section, for $k\in \N$ we write $N_1^*, \ldots, N_k^*$ to denote $N_1, \ldots, N_k\in 2^{\N_0}$ in order to satisfy $N_1^*\ge \ldots\ge N_k^*$.
The trilinear estimate which will be used to obtain the energy estimate and the estimate of the difference is as follows.
\begin{prop}\label{trilin1}
Let $d\in \N$, $\al, \be, \ga\in \R\sm\{0\}$ satisfy $\al-\ga=0$ and $\be + \ga \neq 0$, and $N_1, N_2, N_3\in 2^{\N_0}$.
Then, there exists $\theta=\theta(d)>0$ such that
\begin{align}
&\Bl|\int_0^T\int_{\R^d}(\ove{P_{N_1}u})(P_{N_2}v)(P_{N_3}w)\, dxdt\Br| \label{trilin1_1}\\
&\les  T^{\theta}(N_3^*)^{\fr{d-1}{2}}\|P_{N_1}u\|_{F_{N_1, \al}(T)}\|P_{N_2}v\|_{F_{N_2, \be}(T)}\|P_{N_3}w\|_{F_{N_3, \ga}(T)} \nt
\end{align}
for $T\in (0, 1]$, $u\in F_{\al}^0(T)$, $v\in F_{\be}^0(T)$, and $w\in F_{\ga}^0(T)$.
\end{prop}
We prove this proposition by dividing it into the multi-dimensional case and the one-dimensional case.
We first prove the multi-dimensional case.
\begin{proof}[Proof of Proposition \ref{trilin1} in multi-dimensional case]
In this proof, we denote $L_t^qL_x^p:=L_t^q(\R; L_x^p(\R^d))$.
We divide the interval $[0, T]$ into $N_1^*$ pieces, that is, we put $g_m(t)$ by \eqref{g_m} for $m=1, \ldots, N_1^*$.
Then, note that we have \eqref{g_m1} with $t_m=(m-1)T/N_1^*$ for $N=N_1, N_2, N_3$ and $t\in\R$.
To prove \eqref{trilin1_1}, it suffices to show
\begin{align}
&I
:=\Bl|\int_\R\int_{\R^d}(\ove{g_m(t)P_{N_1}u})(g_m(t)P_{N_2}v)(g_m(t)P_{N_3}w)\, dxdt\Br| \label{trilin1_2}\\
&\les  T^{\theta}(N_1^*)^{-1}(N_3^*)^{\fr{d-1}{2}}\|P_{N_1}u\|_{F_{N_1, \al}(T)}\|P_{N_2}v\|_{F_{N_2, \be}(T)}\|P_{N_3}w\|_{F_{N_3, \ga}(T)} \nt
\end{align}
for $m=1, \ldots, N_1^*$.
We only prove the case of $N_j = N_j^*$ for $j=1,2,3$, but other cases are shown in the same manner.
\\
(i-i) The case of $N_1^* \sim N_2^*\gg N_3^*$.

From Corollaries \ref{str_cor} and \ref{bilin2}, we obtain 
\begin{align*}
I
&\les |\ope{supp}g_m|^{\fr{1}{2}}\|g_m(t)P_{N_1}u\|_{L_t^\infty L_x^2}\|(g_m(t)P_{N_2}v)(g_m(t)P_{N_3}w)\|_{L_t^2 L_x^2} \\
&\les  T^{\fr{1}{2}}(N_1^*)^{-1}(N_3^*)^{\fr{d-1}{2}}\|P_{N_1}u\|_{F_{N_1, \al}(T)}\|P_{N_2}v\|_{F_{N_2, \be}(T)}\|P_{N_3}w\|_{F_{N_3, \ga}(T)} .
\end{align*}
Hence, we obtain the desired estimate by taking $\theta=\fr{1}{2}$.

\noindent
(i-ii) The case of $N_1^* \sim N_3^*$.
\\
(i-iia) The case of $d=2, 3$.

Corollary \ref{str_cor} yields
\begin{align*}
I
&\les |\ope{supp}g_m|^{1-\fr{d}{4}}\|g_mP_{N_1}u\|_{L_t^{12/d} L_x^3}\|g_mP_{N_2}v\|_{L_t^{12/d} L_x^3}\|g_mP_{N_3}w\|_{L_t^{12/d} L_x^3} \\
&\les T^{1-\fr{d}{4}}(N_1^*)^{-1}(N_3^*)^{\fr{d}{4}}\|P_{N_1}u\|_{F_{N_1, \al}(T)}\|P_{N_2}v\|_{F_{N_2, \be}(T)}\|P_{N_3}w\|_{F_{N_3, \ga}(T)} .
\end{align*}
Since $\fr{d}{4}\le \fr{d-1}{2}$, the desired estimate is obtained by taking $\theta= 1-\fr{d}{4}$.
\\
(i-iib) The case of $d\ge 4$.

Let $\theta\in(0, 1)$ be chosen later.
Also, we put $p:=\fr{2d}{d-2(1-\theta)}$ and $q:=\fr{2}{1-\theta}$.
Then, $(q, p)$ is an admissible pair.
Thus, the H\"{o}lder inequality, the Sobolev embedding $H^{(d-4(1-\theta))/2}(\R^d)\hookrightarrow L^{d/(2(1-\theta))}(\R^d)$, and Corollary \ref{str_cor} yield
\begin{align*}
I
&\les |\ope{supp}g_m|^{\theta}\|g_mP_{N_1}u\|_{L_t^{q}L_x^{p}}\|g_mP_{N_2}v\|_{L_t^{q}L_x^{p}}\|g_mP_{N_3}w\|_{L_t^\infty L_x^{d/(2(1-\theta))}}\\
&\les T^{\theta}(N_1^*)^{-\theta}(N_3^*)^{\fr{d-4(1-\theta)}{2}}\|P_{N_1}u\|_{F_{N_1, \al}(T)}\|P_{N_2}v\|_{F_{N_2, \be}(T)}\|P_{N_3}w\|_{F_{N_3, \ga}(T)}.
\end{align*}
Since $N_1^*\sim N_3^*$, we obtain the desired bound by taking $\theta= \fr{1}{2}$.
\end{proof}

Next, we consider the one-dimensional case.
First, we prepare the following lemma.
This means that we can use at least one of the bilinear Strichartz estimate and the modulation estimate in any case.
\begin{lem}\label{modest2}
Let $\al, \be, \ga\in \R\sm\{0\}$ satisfy $\al-\ga=0$ and $\be+\ga\neq 0$.
We take $b\in\R\sm\{-2, -1\}$ such that $\be = \ga(b+1)$.
Let $(\tau_1, \xi_1), (\tau_2, \xi_2), (\tau_3, \xi_3)\in \R\times\R$ satisfy $\tau_1 - \tau_2 -\tau_3 =0$ and $\xi_1-\xi_2-\xi_3=0$.
Then, the followings hold.
\\
(A) If $|\fr{\xi_3}{\xi_2} - \fr{b}{2}|\le \fr{1}{4}|b+2|$, then we have
\[
|\be \xi_2- \ga\xi_3|\ge |\ga||b+2||\xi_2|/4.
\]
(B) If $|\fr{\xi_3}{\xi_2} - \fr{b}{2}|\ge \fr{1}{4}|b+2|$, then we have 
\[
\bl|(\tau_1+\al\xi_1^2) - (\tau_2+\be\xi_2^2) - (\tau_3+\ga\xi_3^2)\br|
\ge |\ga||b+2||\xi_2|^2/2.
\]
\end{lem}
\begin{proof}
We have $|\be \xi_2- \ga\xi_3|= |\ga||\xi_2||\fr{b+2}{2}-(\fr{\xi_3}{\xi_2}-\fr{b}{2})|$,
which yields (A).
Also, $\al-\ga=0$, $\tau_1 - \tau_2 -\tau_3 =0$, and $\xi_1-\xi_2-\xi_3=0$ yields
\begin{equation}\label{modest2_1}
\bl|(\tau_1+\al\xi_1^2) - (\tau_2+\be\xi_2^2) - (\tau_3+\ga\xi_3^2)\br|
= 2|\ga||\xi_2|^2\Bl|\fr{\xi_3}{\xi_2}-\fr{b}{2}\Br|.
\end{equation}
Thus, we obtain (B).
\end{proof}
\begin{proof}[Proof of Proposition \ref{trilin1} in one-dimensional case]
Also in this case, we show \eqref{trilin1_2}.
For $m=1, \ldots, N_1^*$, we set $g_m(t)$ by \eqref{g_m}.

\noindent
(i) The case of $N_1^* \sim N_2^*\gg N_3^*$.

In this case, we can obtain the desired estimate by the same argument as (i-i) in the proof of the multi-dimensional case.
\\
(ii) The case of $N_1^* \sim N_3^*$.

The case of $N_3^*=1$ is easy because we can obtain the desired estimate from the H\"older and Bernstein inequalities.
We consider the case of $N_3^*\in 2^{\N}$.
Let $b$ be what is taken as in Lemma \ref{modest2}.

Let $A:= \{(\xi_2, \xi_3)\in \R^2 \mid |\fr{\xi_3}{\xi_2} - \fr{b}{2}|\le \fr{1}{4}|b+2|\}$, $A^c$ be the complement of $A$, and $I$ be what defined in \eqref{trilin1_2}.
Then, we have
\begin{align}
I&= c\iint_{\substack{\xi_1-\xi_2-\xi_3=0 \\ \tau_1-\tau_2-\tau_3=0}} (\mb{1}_{A}(\xi_2, \xi_3) + \mb{1}_{A^c}(\xi_2, \xi_3)) \nt\\
&\hs{20mm}\times\ove{\cf_{t ,x}[g_mP_{N_1}u](\tau_1, \xi_1)}\cf_{t ,x}[g_mP_{N_2}v](\tau_2, \xi_2)\cf_{t ,x}[g_mP_{N_3}w](\tau_3, \xi_3) \nt\\
&=: I_A + I_{A^c}. \nt
\end{align}
It suffices to consider the estimate of $I_A$ and $I_{A^c}$.
Lemma \ref{modest2} yields $A\subset \{(\xi_2, \xi_3)\in \R^2 \mid |\be \xi_2 -\ga\xi_3|\gtr N_2\}$.
Thus, it holds from Corollary \ref{bilin2}, Plancherel's theorem, and the H\"older inequality that
\begin{align}
|I_A| 
&\les |\ope{supp} g_m|^{1/2}\|g_mP_{N_1}u\|_{L_{t}^\infty L_{x}^2} \nt\\
&\  \times\Bl\|\iint_{\substack{\xi_2+\xi_3=\xi_1 \\ \tau_2+\tau_3=\tau_1}}\mb{1}_{A}(\xi_2, \xi_3)\cf_{t, x}[g_mP_{N_2}v](\tau_2, \xi_2)\cf_{t, x}[g_mP_{N_3}w](\tau_3, \xi_3) \Br\|_{L_{\tau_1}^2L_{\xi_1}^2} \nt\\
&\les T^{1/2}(N_1^*)^{-1/2}N_2^{-1/2}\|P_{N_1}u\|_{F_{N_1, \al}(T)}\|P_{N_2}v\|_{F_{N_2, \be}(T)}\|P_{N_3}w\|_{F_{N_3, \ga}(T)}. \nt
\end{align}
Since $N_1^*\sim N_3^*$, $I_A$ is bounded by the right-hand side of \eqref{trilin1_2} if $\theta=1/2$.

Next, we consider the part of $I_{A^c}$.
Let $n_0\in \N$ satisfy $|\ga||b+2|\gg 2^{-n_0}$ and $M:= 2^{-n_0}(N_1^*)^2\in 2^{\mathbb{Z}}$.
Also, for $\si\in \R\sm\{0\}$ we set
\[
\mc{F}_{t, x}[R_{\ge M}^{\si}f](\tau, \xi):=\mb{1}_{\{|\tau+\si|\xi|^2|\ge M\}}(\tau, \xi)\mc{F}_{t, x}[f](\tau, \xi)
\]
and $R_{< M}^{\si}=1-R_{\ge M}^{\si}$.
From the definition of $X_{N', \si}$, we have
\begin{equation}\label{highmod}
\|R_{\ge M}^{\si}f\|_{L_t^2L_x^2}\les M^{-1/2}\|\cf_{t, x}[f]\|_{X_{N', \si}}
\end{equation}
for $N'\in 2^{\N_0}$, $\si\in\R\sm\{0\}$, and $f\in L_t^2(\R; L_x^2(\R))$ satisfying $\cf_{t, x}[f]\in X_{N', \si}$.
We divide $I_{A^c}$ into $8$ pieces of the integral
\begin{align*}
&I_{A^c}^{*_1, *_2, *_3}
:=c\iint_{\substack{\xi_1-\xi_2-\xi_3=0 \\ \tau_1-\tau_2-\tau_3=0}} \mb{1}_{A^c}(\xi_2, \xi_3)\ove{\cf_{t ,x}[R_{*_1}^{\al}(g_mP_{N_1}u)](\tau_1, \xi_1)} \nt\\
&\hs{25mm}\times\cf_{t ,x}[R_{*_2}^{\be}(g_mP_{N_2}v)](\tau_2, \xi_2)\cf_{t ,x}[R_{*_3}^{\ga}(g_mP_{N_3}w)](\tau_3, \xi_3), \nt
\end{align*}
where $*_1, *_2, *_3$ represent one of $``<M"$ or $``\ge M"$.
Lemma \ref{modest2} (B) and the triangle inequality yield $\max\{|\tau_1+\al|\xi_1|^2|, |\tau_2+\be|\xi_2|^2|, |\tau_3+\ga|\xi_3|^2|\}\ge M$ if $\tau_1- \tau_2- \tau_3=0$, $\xi_1- \xi_2- \xi_3=0$, and $(\xi_2, \xi_3)\in A^c$.
Hence, Plancherel's theorem yields $I_{A^c}^{*_1, *_2, *_3} =0$ when $*_1=*_2=*_3=``<M."$ 
Thus, at least one of $*_1, *_2, *_3$ becomes $``\ge M."$
We may assume that $*_1$ is $``\ge M."$
Other cases are shown in the same manner.
From Plancherel's theorem, the H\"older and Young inequalities, Corollary \ref{str_cor}, and \eqref{highmod}, we obtain
\begin{align*}
&|I_{A^c}^{*_1, *_2, *_3}| \\
&=c\Bl|\int_{\R}\Bl(\int_{\xi_1-\xi_2-\xi_3=0} \mb{1}_{A^c}(\xi_2, \xi_3)\ove{\cf_{x}[R_{*_1}^{\al}(g_mP_{N_1}u)](t, \xi_1)} \nt\\
&\hs{25mm}\times\cf_{x}[R_{*_2}^{\be}(g_mP_{N_2}v)](t, \xi_2)\cf_{x}[R_{*_3}^{\ga}(g_mP_{N_3}w)](t, \xi_3)\Br)\,dt\Br| \nt\\
&\les \bl\|\|\cf_{x}[R_{*_1}^{\al}(g_mP_{N_1}u)]\|_{L_\xi^2}\|\cf_{x}[R_{*_2}^{\be}(g_mP_{N_2}v)]\|_{L_\xi^2}\|\cf_{x}[R_{*_3}^{\ga}(g_mP_{N_3}w)]\|_{L_\xi^1} \br\|_{L_t^1} \nt\\
&\le \|\cf_{x}[R_{*_1}^{\al}(g_mP_{N_1}u)]\|_{L_\tau^2L_\xi^2}\|\cf_{x}[R_{*_2}^{\be}(g_mP_{N_2}v)]\|_{L_\tau^2L_\xi^2}\|\cf_{x}[R_{*_3}^{\ga}(g_mP_{N_3}w)]\|_{L_t^\infty L_\xi^1} \nt\\
&\les |\ope{supp}g_m|^{1/2}N_3^{1/2}\|R_{*_1}^{\al}(g_mP_{N_1}u)\|_{L_t^2 L_x^2}\|g_mP_{N_2}v\|_{L_t^\infty L_x^2}\|g_mP_{N_3}w\|_{L_t^\infty L_x^2} \\
&\les T^{1/2}M^{-1/2}\|P_{N_1}u\|_{F_{N_1, \al}(T)}\|P_{N_2}v\|_{F_{N_2, \be}(T)}\|P_{N_3}w\|_{F_{N_3, \ga}(T)}.
\end{align*}
Since $M\sim (N_1^*)^2$, we obtain the desired estimate.
\end{proof}

\subsection{Some auxiliary estimates}
In this subsection, we prove two quadratic estimates to obtain the energy estimates and estimates of the difference of the solutions to \eqref{CC} and \eqref{K-CC}.
In the following, for $T\in (0, 1]$, $N\in 2^{\N_0}$, $\si_1, \si_2\in \R\sm\{0\}$, and $\mb{f}=(f_1, f_2)\in \til\bff(T):= F_{\si_1}^0(T)\times F_{\si_2}^0(T)$, we denote
\[
\|P_N\mb{f}\|_{\til\bff_{N}(T)}
:= \|P_Nf_1\|_{F_{N, \si_1}(T)} + \|P_Nf_2\|_{F_{N, \si_2}(T)}.
\]

\begin{lem}\label{qu1}
Let $d\in \N$, $\si_1, \si_2\in \R\sm\{0\}$, and $\til s\ge s\ge \fr{1}{2}(d+1)$, then we have
\begin{align}
&N^{\til s-1}\|P_N((\pa_{x_j}f_1)f_2)\|_{L_T^\infty L_x^2} \nt\\
&\les \|\mb{f}\|_{\til\bff^{s}(T)}\sum_{N_1\sim N}N_1^{\til s}\|P_{N_1}\mb{f}\|_{\til\bff_{N_1}(T)}
+ N^{-(s-\fr{d}{2})}\|\mb{f}\|_{\til\bff^{\til s}(T)}\|\mb{f}\|_{\til\bff^{s}(T)} \nt
\end{align}
for $N\in 2^{\N_0}$, $1\le j \le d$, $T>0$ and $\mb{f}=(f_1, f_2)\in \til \bff^{\til s}:=F_{\si_1}^{\til s}(T)\times F_{\si_2}^{\til s}(T)$.
\end{lem}

\begin{proof}
The H\"older and Bernstein inequalities and Corollary \ref{str_cor} yield
\begin{align}
&N^{\til s-1}\|P_N((\pa_{x_j}P_{N_1}f_1)(P_{N_2}f_2))\|_{L_T^\infty L_x^2} \label{pr_qu1_1}\\
&\les N^{\til s-1}\min\{N, N_1, N_2\}^{\fr{d}{2}}N_1\|P_{N_1}f_1\|_{F_{N_1, \si_1}(T)}\|P_{N_2}f_2\|_{F_{N_2, \si_2}(T)}. \nt
\end{align}
(i) The case of $N\sim N_1\gg N_2$. 

In this case, we have $N^{\til s-1}\min\{N, N_1, N_2\}^{\fr{d}{2}}N_1\les N^{\til s}N_2^{\fr{d}{2}}$.
Therefore, if $s>\fr{d}{2}$, then we obtain
\[
\sum_{N\sim N_1\gg  N_2}	(\tm{L.H.S. of \eqref{pr_qu1_1}})
\les \|\mb{f}\|_{\til \bff^{s}(T)}\sum_{N_1\sim N}N_1^{\til s}\|P_{N_1}\mb{f}\|_{\til \bff_{N_1}(T)}.
\]

\noindent
(ii) The case of $N\sim N_2\gg N_1$.

This case is similar to the case of (i). 

\noindent
(iii) The case of $N_1\sim N_2\gtr N$. 

It holds that $N^{\til s-1}\min\{N, N_1, N_2\}^{\fr{d}{2}}N_1\les N^{-(s-\fr{d}{2})}N_1^{\til s}N_2^s$ since $\til s+ \fr{d}{2}\ge 1$.
Thus, if $s> \fr{d}{2}$, then we have
\[
\sum_{N_1\sim N_2\gtr N} (\tm{L.H.S. of \eqref{pr_qu1_1}})
\les N^{-(s-\fr{d}{2})}\|\mb{f}\|_{\til \bff^{\til s}(T)}\|\mb{f}\|_{\til \bff^{s}(T)}.
\]

From  (i)--(iii), we obtain the desired estimate.
\end{proof}
From the similar argument as Lemma \ref{qu1}, the following lemma holds.
We omit the proof.

\begin{lem}\label{qu2}
Let $d\in \N$, $\si_1, \si_2\in \R\sm\{0\}$, $N\in 2^{\N_0}$, and $s>\fr{d}{2}$, then we have
\begin{align}
\|P_N(f_1 f_2)\|_{L_T^\infty L_x^2}
&\les \sum_{\{i_1, i_2\}=\{1, 2\}}\min\Bl\{\sum_{N_{i_1}\sim N}\|P_{N_{i_1}}f_{i_1}\|_{F_{N_{i_1}, \si_{i_1}}(T)}\|f_{i_2}\|_{F_{\si_{i_2}}^{s}(T)}, \nt\\
&\hs{30mm}\sum_{N_{i_1}\sim N}N_{i_1}^s\|P_{N_{i_1}}f_{i_1}\|_{F_{N_{i_1}, \si_{i_1}}(T)}\|f_{i_2}\|_{F_{\si_{i_2}}^{0}(T)}\Br\} \nt\\
&+ N^{-(s-\fr{d}{2})}\min\bl\{\|f_1\|_{F_{\si_1}^{0}(T)}\|f_2\|_{F_{\si_2}^{s}(T)}, \|f_1\|_{F_{\si_1}^{s}(T)}\|f_2\|_{F_{\si_2}^{0}(T)}\br\} \nt
\end{align}
for $N\in 2^{\N_0}$ $T\in (0, 1]$, $f_1\in F_{\si_1}^s(T)$, and $f_2\in F_{\si_2}^s(T)$.
\end{lem}

\section{Energy estimate for a solution}
\subsection{Construction of the modified energy}
To prove Lemma \ref{energy1}, we add the correction terms which depend on the frequency.
In the following, we set $\int f\,dx:= \int_{\R^d}f\,dx$.
For the solution $(u_K, v_K, w_K)$ to \eqref{K-CC}, we can calculate as follows:
\begin{align}
&\fr{1}{2}\fr{d}{dt}\|P_N u_{K}\|_{L^2}^2 
+\ope{Im} \int   (\na\cd  P_Nw_{K}) v_{K} \cd \ove{ P_N u_{K}}\,dx \label{ene1}\\
&= -\ope{Im} \int   ([P_N, v_{K}]\na\cd  w_{K}) \cd \ove{ P_N u_{K}}\,dx, \nt
\end{align}
where $[\cd, \cd]$ denotes the commutator and we used Plancherel's theorem and $J_{\le K} u_K= u_K$.
Similarly, we have
\begin{align}
&\fr{1}{2}\fr{d}{dt}\|P_N v_{K}\|_{L^2}^2
+\ope{Im} \int   (\ove{\na\cd  P_Nw_{K}}) u_{K} \cd \ove{ P_N v_{K}}\,dx  \label{ene2}\\
&=-\ope{Im} \int   ([P_N, u_{K}]\ove{\na\cd  w_{K}}) \cd \ove{ P_N v_{K}}\,dx, \nt
\end{align}
\begin{align}
&\fr{1}{2}\fr{d}{dt}\|P_N w_{K}\|_{L^2}^2 \label{ene3}
 +\ope{Im} \int  (u_{K}\cd \ove{P_N  v_{K}})\ove{\na \cd P_N  w_{K}}\,dx \\
&\quad+\ope{Im} \int  (P_N  u_{K} \cd \ove{v_{K}})\ove{\na \cd P_N  w_{K}}\,dx \nt\\
&=-\ope{Im} \int  (\com(P_N, u_{K}, \ove{v_{K}}))\ove{\na \cd P_N  w_{K}}\,dx, \nt
\end{align}
where $\com(P_N, f, g)$ is defined by
\begin{equation}\label{com}
\com(P_N, f, g):= P_N(f\cd g) - P_Nf\cd g -f\cd P_{N}g.
\end{equation}
Thus, the right-hand side of \eqref{ene3} corresponds to the right-hand sides of \eqref{ene1} and \eqref{ene2}.

The second term on the left-hand side of \eqref{ene1} is canceled out by the third term on the left-hand side of \eqref{ene3} since they have the different sign.
To cancel the second term on the left-hand sides of \eqref{ene2} and \eqref{ene3} which have the same sign, for $N\in 2^{\N}$, we define the correction term $M_N$ by
\begin{equation}\label{ene4}
M_{N}(f, g, h)
:= \ope{Re}\int(f \cd \ove{P_N g}) \ove{\na \cd\Del ^{-1}P_N h}\, dx.
\end{equation}
We note that $\Del^{-1}P_{N}$ is well-defined when $N\in 2^{\N}$.
Also, the derivative in the second term on the left-hand sides of \eqref{ene2} and \eqref{ene3} is harmless in the case of $N=1$.

We have
\begin{align}
&M_{N}(\pt u_{K},  v_{K},  w_{K})  \label{ene5}\\
&=-\al \ope{Im}\int(\Del  u_{K} \cd \ove{P_N  v_{K}})\ove{\na \cd\Del^{-1} P_N  w_{K}}\, dx  \nt\\
&\quad - \ope{Im}\int(J_{\le K}((\na \cd  w_{K})  v_{K}) \cd \ove{P_N  v_{K}})\ove{\na \cd\Del^{-1} P_N  w_{K}}\, dx,  \nt
\end{align}
\begin{align}
&M_{N}(u_{K},  v_{K}, \pt w_{K}) 
-\ga \ope{Im} \int  ( u_{K} \cd \ove{P_N v_{K}})\ove{\na \cd P_N  w_{K}}\,dx \label{ene6}\\
&= - \ope{Im}\int(u_{K}\cd \ove{P_N  v_{K}})\ove{P_N J_{\le K} ( u_{K} \cd \ove{v_{K}})}\, dx. \nt
\end{align}
Also, the integration by parts yields
\begin{align}
&M_{N}(u_{K},  \pt v_{K},  w_{K}) 
-\be \ope{Im} \int  ( u_{K} \cd \ove{P_N v_{K}})\ove{\na \cd P_N  w_{K}}\,dx \label{ene7}\\
&= \be \ope{Im}\int(\Del  u_{K} \cd \ove{P_N  v_{K}})\ove{\na \cd \Del^{-1} P_N  w_{K}}\, dx \nt\\
&\quad\quad +2\be \sum_{j=1}^d \ope{Im}\int(\pa_{x_j}  u_{K} \cd \ove{P_N  v_{K}})\ove{\na \cd \pa_{x_j} \Del^{-1}P_N  w_{K}}\, dx \nt\\
&\quad\quad+ \ope{Im}\int( u_{K}\cd \ove{P_N J_{\le K} (( \ove{\na \cd  w_{K}}) u_{K})})\ove{\na \cd\Del^{-1} P_N  w_{K}}\, dx.  \nt
\end{align}

Therefore, we obtain
\begin{align}
&\fr{d}{dt}M_{N}(u_{K},  v_{K},  w_{K}) 
-(\be+\ga)\ope{Im} \int  ( u_{K} \cd \ove{P_N v_{K}})\ove{\na \cd P_N  w_{K}}\,dx \nt\\
&= (\textrm{R.H.S. of \eqref{ene5}--\eqref{ene7}}), \nt
\end{align}
which yields
\begin{align}
&\fr{1}{2}\fr{d}{dt}\bl(\|(P_N u_K, P_N v_K, P_N w_K)\|_{\ch^0}^2  + \fr{4}{\be + \ga}M_{N}(u_K,  v_K,  w_K)\br) \label{ene9} \\
&= (\textrm{R.H.S. of \eqref{ene1}--\eqref{ene3}}) + \fr{4}{\be+\ga}(\textrm{R.H.S. of \eqref{ene5}--\eqref{ene7}}). \nt
\end{align}

To guarantee the coercivity of the modified energy, we add furthur correction terms as follows:
\begin{equation}\label{ene10}
E_N^{\til s}(\bfu_K) 
:= N^{2\til s} (1 + N^{-1}\til C\|\bfu_K\|_{\ch^{s_0}}^2)\|P_N\bfu_K\|_{\ch^{0}}^2 + \fr{4}{\be + \ga}N^{2\til s}M_{N}(u_K, v_K, w_K) 
\end{equation}
where $\bfu_K=(u_K, v_K, w_K)$, $s_0:= \fr{d}{2} + \fr{1}{2}(s - \fr{1}{2}(d+1))$, and $\til C=\til C(s_0)>0$.
This modified energy has the coercivity.
\begin{lem}\label{coer1}
Let $d\in \N$ and $\til s\ge s> \fr{1}{2}(d+1)$.
Then, there exist $\til C>0$ such that $E_N^{\til s}(\bfu) \gtr N^{2\til s} \|P_N\bfu\|_{\ch^{0}}^2$ for $N\in 2^{\N}$ and $\bfu=(u, v, w)\in \ch^{\til s}$.
\end{lem}
\begin{proof}
It suffices to show
\begin{equation}\label{pr_coer1}
\fr{4}{|\be +\ga|}N^{2\til s}|M_{N}(u, v, w)|-\fr{1}{2}N^{2\til s} \|P_N \bfu\|_{\ch^0}^2
\les N^{2\til s-1}\|\bfu\|_{\ch^{s_0}}^2\|P_N\bfu\|_{\ch^{0}}^2.
\end{equation}
Then, the H\"older inequality and the Sobolev embedding $H^{s_0}(\R^d)\hookrightarrow L^\infty(\R^d)$ yields $|M_N(u, v, w)|\les N^{-1}\|\bfu\|_{\ch^{s_0}}\|P_{N} \bfu\|_{\ch^{0}}^2$.
Therefore, Young's inequality $\|\bfu\|_{\ch^{s_0}}\le \fr{1}{2C} + \fr{C}{2}\|\bfu\|_{\ch^{s_0}}^2$ for $C>0$ leads to the desired bound \eqref{pr_coer1}.
\end{proof}

Before showing  Lemma \ref{energy1}, we collect the following lemma.
\begin{lem}\label{s_0norm}
We have
\[
\fr{1}{2}\fr{d}{dt}\|\bfu_K\|_{\ch^{s_0}}^2
\les \|\bfu_K\|_{\ch^s}^3
\]
for any $s>\fr{1}{2}(d+1)$, $T>0$, $t\in [0, T]$, $K\in 2^{\N}\cup\{\infty\}$ and a solution $\bfu_K=(u_K, v_K, w_K)\in C([0, T]; \ch^{s+\fr{3}{2}})$ to \eqref{K-CC}.
\end{lem}
\begin{proof}
We only consider the time derivative of $\|u_K\|_{\ch^{s_0}}^2$, but the same argument holds for $\|v_K\|_{\ch^{s_0}}^2$ and $\|w_K\|_{\ch^{s_0}}^2$.
We obtain
\begin{align}
\fr{1}{2}\fr{d}{dt}\|u_K\|_{H^{s_0}}^2
&= -\ope{Im}\int\lan\na\ran^{s_0}((\na\cd w_K)v_K)\cd\ove{\lan\na\ran^{s_0}u_K}\,dx \nt\\
&\les \|(\na\cd w_K)v_K\|_{H^{s_0-\fr{1}{2}}}\|u_K\|_{H^{s_0+\fr{1}{2}}} \nt\\
&\les \|(\na\cd w_K)\|_{H^{s-1}}\|v_K\|_{H^{s}}\|u_K\|_{H^{s}}, \nt
\end{align}
where we used the inequality $\|fg\|_{H^{s_0-\fr{1}{2}}}\les \|f\|_{H^{s-1}}\|g\|_{H^{s}}$.
This completes the proof.
\end{proof}

\subsection{Proof of Lemma \ref{energy1}}
In the following, we may omit the subscript $K$ since it is harmless.
From Lemma \ref{coer1}, \eqref{ene9}, and \eqref{ene10}, we obtain
\begin{align}
&N^{2\til s}\|P_N \bfu(t)\|_{\ch^0}^2
\les  E_N^{\til s}(\bfu(t)) \label{pr_ene1}\\
&=  E_N^{\til s}(\bfu(0)) + \int_0^t\fr{d}{dt'}E_N^{\til s}(\bfu(t'))\,dt'\nt\\
&\les N^{2\til s}(1+\|\bfu(0)\|_{\ch^{s_0}}^2)\|P_N \bfu(0)\|_{\ch^0}^2 + N^{-1}\|P_N \bfu(t)\|_{\ch^{s_0}}^2\Bl|\int_0^t \fr{d}{dt'}\|\bfu\|_{\ch^{s_0}}^2\,dt'\Br|\nt\\
&\quad +(1+\|\bfu(t)\|_{\ch^s}^2)N^{2\til s}\Bl|\int_0^t (\textrm{R.H.S. of \eqref{ene1}--\eqref{ene3}})\,dt'\Br| \nt\\
&\quad + N^{2\til s}\Bl|\int_0^t (\textrm{R.H.S. of \eqref{ene5}--\eqref{ene7}})\,dt'\Br| \nt\\
&\quad + N^{2\til s-1}\|\bfu(t)\|_{\ch^s}^2\Bl|\int_0^t\int_{\R^d}(\ove{\na\cd  P_Nw}) u \cd \ove{ P_N v}\,dxdt'\Br|\nt
\end{align}
for $N\in 2^{\N}$ and $t\in [0, T]$.
From Lemma \ref{s_0norm} and the embedding $F_{\si}^{s}(T)\hookrightarrow L_T^\infty H^{s}$ for $\si\in \R\sm\{0\}$, the second term on the right-hand side of \eqref{pr_ene1} is bounded by
$N^{-1}T\|\bfu\|_{\bff^s(T)}^5$.
Also, from Proposition \ref{trilin1}, the last term on the right-hand side of \eqref{pr_ene1} is bounded by
$T\|\bfu\|_{\bff^s(T)}N^{2\til s}\|P_N \bfu\|_{\bff_N(T)}^2$,
where $\|P_N \bfu\|_{\bff_{N}}$ is defined by
\begin{equation}\label{F_N}
\|P_N \bfu\|_{\bff_{N}(T)}^2:= \|P_N u\|_{F_{N, \al}(T)}^2 + \|P_N v\|_{F_{N, \be}(T)}^2 + \|P_N w\|_{F_{N, \ga}(T)}^2.
\end{equation}
Thus, it suffices to show
\begin{align}
&N^{2\til s}\Bl|\int_0^t(\tm{R.H.S. of \eqref{ene1}--\eqref{ene3}, \eqref{ene5}--\eqref{ene7}})\, dt'\Br| \label{pr_ene3}\\
&\les T^{\theta}(\|\bfu\|_{\bff^s(T)}+ \|\bfu\|_{\bff^s(T)}^2)\Bl(\sum_{N_1\sim N}N_1^{\til s}\|P_{N_1} \bfu\|_{\bff_{N_1}(T)}\Br)^2 \nt\\
&\quad +  T^{\theta}(\|\bfu\|_{\bff^s(T)}+ \|\bfu\|_{\bff^s(T)}^2)\|\bfu\|_{\bff^{\til s}(T)}\cd N^{\til s-\del}\|P_N \bfu\|_{\bff_{N}(T)} \nt
\end{align}
for some $\theta\in (0, 1]$, $\del>0$ and all $t\in [0, T]$ if $\til s\ge s>\fr{d+1}{2}$.
Indeed, if we obtain \eqref{pr_ene3}, then the desired estimate follows by taking the supremum over $t\in [0, T]$ on \eqref{pr_ene1} and summing them up over $N\in 2^{\N}$.
We note that the integral over $[0, t]$ for the second term on the left-hand side of \eqref{ene2} is bounded by the first term on the right-hand side of \eqref{pr_ene3} in the case of $N=1$.

It suffices to show the case of $t=T$ in \eqref{pr_ene3}.
Also, we only consider the estimate of the first term on the right-hand side of \eqref{ene1} and the third term on the right-hand side of \eqref{ene7}.
The first term on the right-hand sides of \eqref{ene2}, \eqref{ene3}, and \eqref{ene5}, and the first and second terms on the right-hand side of \eqref{ene7} are estimated in the same manner as the first term on the right-hand side of \eqref{ene1}.
Also, the second term on the right-hand side of \eqref{ene5} and the first term on the right-hand side of \eqref{ene6} are estimated by the similar argument as the third term on the right-hand side of \eqref{ene7}.

We first consider the estimate of the third term on the right-hand side of \eqref{ene7}.
From Lemma \ref{qu1} and the H\"older and Bernstein inequality, we have 
\begin{align}
&N^{2\til s}\Bl|\int_0^T\int_{\R^d}(u\cd P_{N}((\ove{\na\cd w})u)) \ove{\na\cd \Del^{-1} P_{N}w}\,dxdt\Br| \label{pr_ene5} \\
&\le T\|u\|_{L_T^\infty L_x^\infty}\cd N^{\til s -1}\|P_{N}((\ove{\na\cd w})u)\|_{L_T^\infty L_x^2}\cd N^{\til s +1}\|\na\cd \Del^{-1} P_{N}w\|_{L_T^\infty L_x^2} \nt \\
&\les T\|u\|_{F_{\al}^{s}(T)}\Bl(\sum_{N_2\sim N}N_2^{\til s}\|P_{N_2} w\|_{F_{N_2, \ga}(T)}\Br)\|u\|_{F_{\al}^{s}(T)}N^{\til s}\|P_{N} w\|_{F_{N, \ga}(T)} \nt \\
&\quad + TN^{-1/2}\|u\|_{F_{\al}^{s}(T)}\|\bfu\|_{\bff^{\til s}(T)}\|\bfu\|_{\bff^{s}(T)}N^{\til s}\|P_{N} w\|_{F_{N, \ga}(T)} \nt
\end{align}
if $\til s\ge s\ge \fr{1}{2}(d+1)$, where we used $\|u\|_{L_T^\infty L_x^\infty}\les \sum_{N\in 2^{\N}}N^{d/2}\|P_N u\|_{L_T^\infty L_x^2}\les \|u\|_{F_{\al}^{s}(T)}$.
The right-hand side of \eqref{pr_ene5} is bounded by the right-hand side of \eqref{pr_ene3} with $\theta=1$, $\del=\fr{1}{2}$.

Next, we consider the estimate of the first term on the right-hand side of \eqref{ene1}.
We can decompose as
\begin{align}
&\int   ([P_N, v]\na\cd  w) \cd \ove{ P_N u}\,dx  = \hs{-4mm}\sum_{N_1\sim N_2\gg N}\int P_N((\na\cd  P_{N_1}w) P_{N_2}v) \cd \ove{ P_N u}\,dx   \label{pr_ene6}\\
&\hs{15mm}+\Bl(\sum_{N_1\les N_2\sim N} +\sum_{N_2 \ll N_1\sim N}\Br)\int   ([P_N, P_{N_2}v]\na\cd P_{N_1}w) \cd \ove{ P_N u}\,dx. \nt
\end{align}
From  Proposition \ref{trilin1}, we have
\begin{align}
& N^{2\til s}\sum_{N_1\sim N_2\gg N}\Bl|\int_0^T\int_{\R^d} P_N((\na\cd  P_{N_1}w) P_{N_2}v) \cd \ove{ P_N u}\,dxdt\Br|  \label{pr_ene6.5}\\
&\les T^{\theta}\hs{-2mm}\sum_{N_1\sim N_2\gg N}\hs{-2mm}N_1N^{2\til s+ \fr{1}{2}(d-1)}\|P_{N}\bfu\|_{\bff_N(T)}\|P_{N_2}\bfu\|_{\bff_{N_2}(T)}\|P_{N_1}\bfu\|_{\bff_{N_1}(T)} \nt\\
&\les T^{\theta}\hs{-2mm}\sum_{N_1\sim N_2\gg N}\hs{-2mm}N_1^{s}N_2^{\til s}N^{-(s-\fr{1}{2}(d+1)) + \til s}\|P_{N}\bfu\|_{\bff_N(T)}\|P_{N_2}\bfu\|_{\bff_{N_2}(T)}\|P_{N_1}\bfu\|_{\bff_{N_1}(T)} \nt
\end{align}
for some $\theta=\theta(d)>0$.
Hence, this term is bounded by the second term on the right-hand side of \eqref{pr_ene3} with $\del=s-\fr{d+1}{2}$  if $s>\fr{1}{2}(d+1)$.
We consider the second term on the right-hand side of \eqref{pr_ene6}.
We put $u=(u_1, \ldots, u_d)$, $v=(v_1, \ldots, v_d)$, $w=(w_1, \ldots, w_d)$.
From the mean value theorem, we have
\begin{align}
&\int_0^T\int_{\R^d}[P_N, P_{N_2}v](\na\cd P_{N_1}w) \cd \ove{P_N u}\,dxdt \label{pr_ene7}\\
&=-\sum_{j=1}^d\int_0^T\int_{\R^d} \int_0^1\int_{\R^d}\cf^{-1}[\psi_N](y)y\cd (\na P_{N_2}v_j)(t, x-\rho y) \nt\\
&\hs{50mm}\times (\na\cd  P_{N_1}w)(t, x-y)\ove{P_N u_j}\,dyd\rho dxdt \nt\\
&=-\sum_{j, k=1}^d \int_0^1\int_{\R^d}\cf^{-1}[\psi_N](y)y_k \nt\\
&\times \Bl(\int_0^T\hs{-1mm}\int_{\R^d} (\pa_{x_k} P_{N_2}v_j)(t, x-\rho y)(\na\cd P_{N_1}w)(t, x-y)\ove{P_N u_j}\,dxdt\Br)dyd\rho. \nt
\end{align}
Here, Proposition \ref{trilin1} and the invariance of the $F_{N, \si}(T)$-norm for the translation with respect to the spatial variables yield
\begin{align}
&\sup_{\rho\in (0, 1),\ y\in \R^d}\Bl|\int_0^T\int_{\R^d} (\pa_{x_k} P_{N_2}v_j)(t, x-\rho y)(\na\cd  P_{N_1}w)(t, x-y)\ove{P_N u_j}\,dxdt\Br| \label{pr_ene8}\\
&\les T^{\theta}\min\{N, N_1, N_2\}^{\fr{d-1}{2}} N_1 N_2\|P_{N}u\|_{F_{N, \al}(T)}\|P_{N_2}v\|_{F_{N_2, \be}(T)}\|P_{N_1}w\|_{F_{N_1, \ga}(T)} \nt
\end{align}
for $1\le j, k\le d$ and some $\theta=\theta(d)>0$.
Since $\int_{\R^d}|\cf^{-1}[\psi_N](y)||y_k| \,dy\les N^{-1}$ for $1\le k\le d$, we obtain
\begin{align}
&N^{2\til s} \hs{-4mm} \sum_{N_1\les N_2\sim N}\Bl|\int_0^T\int_{\R^d}  [P_N, P_{N_2}v](\na\cd  P_{N_1}w) \cd \ove{ P_N u}\,dxdt\Br| \nt\\
&\les T^{\theta}N^{2\til s-1} \hs{-4mm} \sum_{N_1\les N_2\sim N} \hs{-4mm} N_1^{(d+1)/2}N_2\|P_{N}\bfu\|_{\bff_N(T)}\|P_{N_2}\bfu\|_{\bff_{N_2}(T)}\|P_{N_1}\bfu\|_{\bff_{N_1}(T)} \nt\\
&\les T^{\theta}\|\bfu\|_{\bff^s(T)}N^{\til s}\|P_{N}\bfu\|_{\bff_N(T)}\ \hs{-2mm} \sum_{N_2\sim N}N_2^{\til s}\|P_{N_2}\bfu\|_{\bff_{N_2}(T)} \nt
\end{align}
if $s>\fr{d+1}{2}$.
The estimate of the part of $N_2\ll N_1\sim N$ is similar to the part of $N_1\ll N_2\sim N$.
This yields the bound for the first term on the right-hand side of \eqref{ene1}. 
We note that we decompose the right-hand side of \eqref{ene3} as
\begin{align}
&\int  (P_N(u\cd\ove{v}) - P_{N}u\cd \ove{v}- u\cd \ove{P_{N}v})\ove{\na \cd P_N w}\,dx  \label{pr_ene10}\\
 &= \sum_{N_1\sim N_2\gg N}\int P_N(P_{N_1}u\cd\ove{P_{N_2}v})\ove{\na\cd P_N w}\,dx  \nt\\
&\quad+\sum_{N_1\les N_2\sim N} \int ([P_N, P_{N_1}u ]\cd \ove{P_{N_2}v} - P_NP_{N_1}u\cd \ove{P_{N_2}v}) \ove{\na\cd P_N w}\,dx \nt\\
&\quad+\sum_{N_2 \ll N_1\sim N}\int ([P_N, \ove{P_{N_2}v}]\cd P_{N_1}u - P_{N_1}u\cd \ove{P_{N}P_{N_2}v}) \ove{\na\cd P_N w}\,dx, \nt
\end{align}
where $[P_N, f]\cd g:=P_N(f\cd g)-f\cd P_N g$.
Then, we can obtain the desired estimate by the same argument as the right-hand side of \eqref{ene1}.
Thus, we obtain \eqref{pr_ene3}, which concludes the proof of Lemma \ref{energy1}.

\section{Estimate for the difference}
\subsection{Correction terms in the case of the difference}
Throughout this section, we assume $K_1\le K_2$.
For $j=1, 2$, let $(u_j, v_j, w_j)$ be a solution to \eqref{K-CC} with $K=K_j$.
Also, let $(\til u, \til v, \til w):= (u_1 - u_2, v_1 - v_2, w_1 - w_2)$.
Then, we obtain the equation of the difference \eqref{K-CC_diff}.
Similarly to \eqref{ene1}, we have
\begin{align}
&\fr{1}{2}\fr{d}{dt}\|P_N \til u\|_{L^2}^2 
-\ope{Im} \int   (\na\cd  P_N \til w) \til v \cd \ove{ P_N \til u}\,dx + \ope{Im} \int   (\na\cd  P_N \til w) v_1 \cd \ove{ P_N \til u}\,dx \label{diff1}\\
&= - \ope{Im} \int   (\na\cd  P_N w_1) \til v \cd \ove{ P_N \til u}\,dx + \ope{Im} \int   (\na\cd  P_N w_1) v_1 \cd \ove{ P_N J_{(K_1, K_2]}\til u}\,dx \nt\\
& +\ope{Im} \int   ([P_N, \til v]\na\cd  \til w) \cd \ove{ P_N \til u}\,dx - \ope{Im} \int   ([P_N, v_1]\na\cd  \til w) \cd \ove{ P_N \til u}\,dx \nt\\
& - \ope{Im} \int   ([P_N, \til v]\na\cd w_1) \cd \ove{ P_N \til u}\,dx + \ope{Im} \int   ([P_N, v_1]\na\cd  w_1) \cd \ove{ P_N J_{(K_1, K_2]}\til u}\,dx, \nt
\end{align}
\begin{align}
&\fr{1}{2}\fr{d}{dt}\|P_N \til v\|_{L^2}^2 
- \ope{Im} \int   (\ove{\na\cd  P_N \til w}) \til u \cd \ove{ P_N \til v}\,dx + \ope{Im} \int   (\ove{\na\cd  P_N \til w}) u_1 \cd \ove{ P_N \til v}\,dx \label{diff2}\\
&= - \ope{Im} \int   (\ove{\na\cd  P_N w_1}) \til u \cd \ove{ P_N \til v}\,dx + \ope{Im} \int   (\ove{\na\cd  P_N w_1}) u_1 \cd \ove{ P_N J_{(K_1, K_2]}\til v}\,dx \nt\\
& +\ope{Im} \int   ([P_N, \til u]\ove{\na\cd  \til w}) \cd \ove{ P_N \til v}\,dx - \ope{Im} \int   ([P_N, u_1]\ove{\na\cd  \til w}) \cd \ove{ P_N \til v}\,dx \nt\\
& - \ope{Im} \int   ([P_N, \til u]\ove{\na\cd w_1}) \cd \ove{ P_N \til v}\,dx + \ope{Im} \int   ([P_N, u_1]\ove{\na\cd  w_1}) \cd \ove{ P_N J_{(K_1, K_2]}\til v}\,dx, \nt
\end{align}
\begin{align}
&\fr{1}{2}\fr{d}{dt}\|P_N \til w\|_{L^2}^2 
- \ope{Im} \int  (P_N  \til u \cd \ove{\til v})\ove{\na \cd P_N  \til w}\,dx 
- \ope{Im} \int  (\til u \cd \ove{P_N \til v})\ove{\na \cd P_N  \til w}\,dx \label{diff3}\\
&\quad + \ope{Im} \int  (P_N  \til u \cd \ove{v_1})\ove{\na \cd P_N  \til w}\,dx
+ \ope{Im} \int  (u_1 \cd \ove{P_N \til v})\ove{\na \cd P_N  \til w}\,dx  \nt\\
&= - \ope{Im} \int  (\til u \cd \ove{P_N v_1})\ove{\na \cd P_N  \til w}\,dx
- \ope{Im} \int  (P_N  u_1 \cd \ove{\til v})\ove{\na \cd P_N  \til w}\,dx \nt\\
& + \ope{Im} \int  (P_N u_1 \cd \ove{v_1})\ove{\na \cd P_N J_{(K_1, K_2]} \til w}\,dx
  + \ope{Im} \int  (u_1 \cd \ove{P_N v_1})\ove{\na \cd P_N J_{(K_1, K_2]} \til w}\,dx \nt\\
& + \ope{Im} \int  (\com(P_N, \til u, \ove{\til v}))\ove{\na \cd P_N \til w}\,dx 
 - \ope{Im} \int  (\com(P_N, \til u, \ove{v_1}))\ove{\na \cd P_N \til w}\,dx \nt\\
& + \ope{Im} \int  (\com(P_N, u_1, \ove{\til v}))\ove{\na \cd P_N \til w}\,dx 
+  \ope{Im} \int  (\com(P_N, u_1, \ove{v_1}))\ove{\na \cd P_N J_{(K_1, K_2]}\til w}\,dx, \nt
\end{align}
where $\com(P_N, f, g)$ is defined in \eqref{com}.
The second and third terms on the left-hand sides of \eqref{diff1} and the second and fourth terms on the left-hand side of \eqref{diff3} are canceled each other.
To cancel the second and third terms on the left-hand side of \eqref{diff2} and the third and fifth term on the left-hand side of \eqref{diff3}, we add the correction terms $M_{N}( \til u,   \til v,   \til w)$ and $M_{N}( u_1,   \til v,   \til w)$, where $M_N$ is defined by \eqref{ene4}.

Similar calculations as \eqref{ene5}--\eqref{ene7} yield
\begin{align}
M_N(\pt \til u, \til v, \til w)  
&= -\al \ope{Im}\int(\Del \til u\cd \ove{P_{N}  \til v})\ove{\na\cd \Del^{-1} P_{N} \til w}\,dx \label{diff4}\\
&\quad + \ope{Im}\int(J_{\le K_2}((\na\cd \til w)\til v)\cd \ove{P_{N}  \til v})\ove{\na\cd \Del^{-1} P_{N} \til w}\,dx \nt\\
&\quad - \ope{Im}\int(J_{\le K_2}((\na\cd \til w) v_1)\cd \ove{P_{N}  \til v})\ove{\na\cd \Del^{-1} P_{N} \til w}\,dx \nt\\
&\quad - \ope{Im}\int(J_{\le K_2}((\na\cd w_1)\til v)\cd \ove{P_{N}  \til v})\ove{\na\cd \Del^{-1} P_{N} \til w}\,dx \nt\\
&\quad + \ope{Im}\int(J_{(K_1, K_2]}((\na\cd w_1)v_1)\cd \ove{P_{N}  \til v})\ove{\na\cd \Del^{-1} P_{N} \til w}\,dx, \nt
\end{align}
\begin{align}
&M_N(\til u, \pt \til v, \til w)  - \be \ope{Im}\int(\til u\cd \ove{P_{N} \til v})\ove{\na\cd P_{N} \til w}\,dx \label{diff5}\\
&= \be\ope{Im}\int(\Del \til u\cd \ove{P_{N} \til v})\ove{\na\cd \Del^{-1} P_{N} \til w}\,dx \nt\\
&\quad\quad + 2\be\sum_{j=1}^d\ope{Im}\int(\pa_{x_j} \til u\cd \ove{P_{N} \til v})\ove{\na\cd \Del^{-1} \pa_{x_j}P_{N} \til w}\,dx \nt\\
&\quad - \ope{Im}\int( \til u\cd \ove{ P_{N}J_{\le K_2}((\ove{\na\cd \til w})\til u)})\ove{\na\cd \Del^{-1} P_{N}\til w}\,dx \nt\\
&\quad + \ope{Im}\int( \til u\cd \ove{ P_{N}J_{\le K_2}((\ove{\na\cd \til w})u_1)})\ove{\na\cd \Del^{-1} P_{N}\til w}\,dx \nt\\
&\quad + \ope{Im}\int( \til u\cd \ove{ P_{N}J_{\le K_2}((\ove{\na\cd w_1})\til u)})\ove{\na\cd \Del^{-1} P_{N}\til w}\,dx \nt\\
&\quad - \ope{Im}\int( \til u\cd \ove{ P_{N}J_{(K_1, K_2]}((\ove{\na\cd w_1})u_1)})\ove{\na\cd \Del^{-1} P_{N}\til w}\,dx, \nt
\end{align}
\begin{align}
&M_N(\til u, \til v, \pt\til w)  - \ga\ope{Im}\int(\til u\cd \ove{P_{N} \til v})\ove{\na\cd P_{N} \til w}\,dx \label{diff6}\\
&= \ope{Im}\int( \til u\cd \ove{P_{N} \til v})\ove{P_{N}J_{\le K_2}(\til u\cd \ove{ \til v})}\,dx  - \ope{Im}\int( \til u\cd \ove{P_{N} \til v})\ove{P_{N}J_{\le K_2}(\til u\cd \ove{v_1})}\,dx \nt\\
& - \ope{Im}\int( \til u\cd \ove{P_{N} \til v})\ove{P_{N}J_{\le K_2}(u_1\cd \ove{ \til v})}\,dx + \ope{Im}\int( \til u\cd \ove{P_{N} \til v})\ove{P_{N}J_{(K_1, K_2]}(u_1\cd \ove{v_1})}\,dx. \nt
\end{align}
For $M_N(u_1, \til v, \til w)$, we can similarly calculate as follows:
\begin{align}
M_N(\pt u_1, \til v, \til w)  \label{diff7}
&= -\al \ope{Im}\int(\Del u_1\cd \ove{P_{N}  \til v})\ove{\na\cd \Del^{-1}P_{N} \til w}\,dx \\
&\quad - \ope{Im}\int(J_{\le K_1}((\na\cd w_1)v_1)\cd \ove{P_{N}  \til v})\ove{\na\cd \Del^{-1}P_{N} \til w}\,dx, \nt
\end{align}
\begin{align}
&M_N(u_1, \pt\til v, \til w)  - \be\ope{Im}\int(u_1\cd \ove{P_{N} \til v})\ove{\na\cd P_{N} \til w}\,dx \label{diff8}\\
&=  \be\ope{Im}\int(\Del u_1\cd \ove{P_{N} \til v})\ove{\na\cd \Del^{-1} P_{N} \til w}\,dx \nt\\
&\quad\quad + 2\be\sum_{j=1}^d\ope{Im}\int(\pa_{x_j} u_1\cd \ove{P_{N} \til v})\ove{\na\cd \Del^{-1} \pa_{x_j}P_{N} \til w}\,dx \nt\\
&\quad - \ope{Im}\int( u_1\cd \ove{ P_{N}J_{\le K_2}((\ove{\na\cd \til w})\til u)})\ove{\na\cd \Del^{-1} P_{N}\til w}\,dx \nt\\
&\quad + \ope{Im}\int( u_1\cd \ove{ P_{N}J_{\le K_2}((\ove{\na\cd \til w})u_1)})\ove{\na\cd \Del^{-1} P_{N}\til w}\,dx \nt\\
&\quad + \ope{Im}\int( u_1\cd \ove{ P_{N}J_{\le K_2}((\ove{\na\cd w_1})\til u)})\ove{\na\cd \Del^{-1} P_{N}\til w}\,dx \nt\\
&\quad - \ope{Im}\int( u_1\cd \ove{ P_{N}J_{(K_1, K_2]}((\ove{\na\cd w_1})u_1)})\ove{\na\cd \Del^{-1} P_{N}\til w}\,dx, \nt
\end{align}
\begin{align}
&M_N(u_1, \til v, \pt \til w)  - \ga\ope{Im}\int(u_1\cd \ove{P_{N} \til v})\ove{\na\cd P_{N} \til w}\,dx \label{diff9}\\
&=  \ope{Im}\int(u_1\cd \ove{P_{N} \til v})\ove{P_{N}J_{\le K_2}(\til u\cd \ove{ \til v})}\,dx  - \ope{Im}\int(u_1\cd \ove{P_{N} \til v})\ove{P_{N}J_{\le K_2}(\til u\cd \ove{v_1})}\,dx \nt\\
& - \ope{Im}\int(u_1\cd \ove{P_{N} \til v})\ove{P_{N}J_{\le K_2}(u_1\cd \ove{ \til v})}\,dx  + \ope{Im}\int(u_1\cd \ove{P_{N} \til v})\ove{P_{N}J_{(K_1, K_2]}(u_1\cd \ove{v_1})}\,dx. \nt
\end{align}
Therefore, we obtain
\begin{align}
&\fr{1}{2}\fr{d}{dt}\bl(\|(P_N \til u, P_N \til v, P_N \til w)\|_{\ch^0}^2  - \fr{4}{\be + \ga}(M_{N}(\til u,  \til v,  \til w) - M_{N}( u_1,  \til v,  \til w))\br) \label{diff10} \\
&= (\textrm{R.H.S. of \eqref{diff1}--\eqref{diff3}})+\fr{4}{\be+\ga}\times (\textrm{R.H.S. of \eqref{diff4}--\eqref{diff9}}). \nt
\end{align}

We construct the modified energies in the case of the diiference as follows:
\begin{align}\label{diff11}
\til E_N^{r}(\til \bfu) 
&:= N^{2r} (1+ \til CN^{-1}(\|\bfu_1\|_{\ch^{s_0}}^2 + \|\bfu_2\|_{\ch^{s_0}}^2)\|P_N\til \bfu\|_{\ch^{0}}^2 \\
&\quad - \fr{4}{\be + \ga}N^{2r}(M_{N}(\til u, \til v, \til w)-M_{N}(u_1, \til v, \til w)), \nt
\end{align}
where $s> \fr{1}{2}(d+1)$, $s_0=\fr{d}{2}+ \fr{1}{2}(s-\fr{1}{2}(d+1))$, $\til C=\til C(s_0, d)>0$, and $r=0, s$.
From this definition, we obtain the coercivity of the modified energy.
We omit the proof because the proof is the same as Lemma \ref{coer1}.
\begin{lem}\label{coer2}
Let $d\in \N$, $ s> \fr{1}{2}(d+1)$, and $r=0, s$.
Then, there exist $\til C>0$ such that $\til E_N^{r}(\til \bfu) \gtr N^{2r} \|P_N\bfu\|_{\ch^{0}}^2$ for $N\in 2^{\N}$ and $(u, v, w)\in \ch^{s}$.
\end{lem}

\subsection{Proof of Lemma \ref{dene2}}
From Lemma \ref{coer2}, \eqref{diff10}, and \eqref{diff10}, it holds
\begin{align}
&N^{2s}\|P_N \til \bfu(t)\|_{\ch^0}^2 \label{pr_dene2_1}\\
&\les N^{2s}(1+\|\bfu_1(0)\|_{\ch^{s_0}}^2 + \|\bfu_2(0)\|_{\ch^{s_0}}^2)\|P_N \til \bfu(0)\|_{\ch^0}^2 \nt\\
&\quad+ N^{-1}\|P_N \bfu(t)\|_{\ch^{s_0}}^2\Bl|\int_0^t \fr{d}{dt'}(\|\bfu_1\|_{\ch^{s_0}}^2 + \|\bfu_2\|_{\ch^{s_0}}^2)\,dt'\Br|\nt\\
&\quad +(1+\|\bfu_1(t)\|_{\ch^{s_0}}^2 + \|\bfu_2(t)\|_{\ch^{s_0}}^2)N^{2s}\Bl|\int_0^t (\textrm{R.H.S. of \eqref{diff1}--\eqref{diff3}})\,dt'\Br| \nt\\
&\quad + N^{2s}\Bl|\int_0^t (\textrm{R.H.S. of \eqref{diff4}--\eqref{diff9}})\,dt'\Br| \nt\\
&\quad+ N^{2s-1}(\|\bfu_1(t)\|_{\ch^{s_0}}^2 + \|\bfu_2(t)\|_{\ch^{s_0}}^2) \nt\\
&\quad\quad \times \Bl|\int_0^t\int_{\R^d}\Bl((\ove{\na\cd  P_N \til w}) \til u \cd \ove{ P_N \til v}- (\ove{\na\cd  P_N \til w}) u_1 \cd \ove{ P_N \til v}\Br)\,dxdt'\Br|\nt
\end{align}
for $N\in 2^{\N}$ and $t\in [0, T]$.
From Lemma \ref{s_0norm}, the second term on the right-hand side of \eqref{pr_dene2_1} is bounded by $N^{-1}T(\|\bfu_1\|_{\bff^s(T)}^3 + \|\bfu_2\|_{\bff^s(T)}^3)\|P_N \til \bfu\|_{\bff_N(T)}^2$, where $\|P_N\til \bfu\|_{\bff_N(T)}$ is defined by \eqref{F_N}.
Also, Proposition \ref{trilin1} yields that the last term on the right-hand side of \eqref{pr_dene2_1} is evaluated by $T(\|\bfu_1\|_{\bff^s(T)} + \|\bfu_2\|_{\bff^s(T)})N^{2s}\|P_N \bfu\|_{\bff_N(T)}^2$.

In the following, to simplify the notation, for a normed space $X$, we denote 
\begin{equation*}
\|(F, G)\|_{[X]^2}:=\|F\|_{X}+\|G\|_{X}.
\end{equation*}
For Lemma \ref{dene2}, it suffices to show
\begin{align}
&N^{2s}\Bl|\int_0^t(\tm{R.H.S. of }\eqref{diff1}\tm{--}\eqref{diff9})\, dt'\Br| \label{pr_dene2_3}\\
&\les T^{\theta}(\|(\bfu_1, \bfu_2)\|_{[\bff^s(T)]^2} + \|(\bfu_1, \bfu_2)\|_{[\bff^s(T)]^2}^2)\Bl(\sum_{N_1\sim N}N_1^{s}\|P_{N_1} \mb{\til u}\|_{\bff_{N_1}(T)}\Br)^2 \nt\\
&  + T^{\theta}\|\mb{\til u}\|_{\bff^s(T)}(1 +\|(\bfu_1, \bfu_2)\|_{[\bff^s(T)]^2}) \nt\\
&\hs{15mm}\times \sum_{N_1\sim N}N_1^{s}\|P_{N_1} \mb{\til u}\|_{\bff_{N_1}(T)}\sum_{N_2\sim N}N_2^{s}\|(P_{N_2}\bfu_1, P_{N_2}\bfu_2)\|_{[\bff_{N_2}(T)]^2} \nt\\
& +  T^{\theta}(\|(\bfu_1, \bfu_2)\|_{[\bff^s(T)]^2} +\|(\bfu_1, \bfu_2)\|_{[\bff^s(T)]^2}^2)\|\mb{\til u}\|_{\bff^{s}(T)}\cd  N^{s-\del}\|P_N \mb{\til u}\|_{\bff_{N}(T)} \nt\\
& + T^{\theta}\|\mb{\til u}\|_{\bff^0(T)}\cd  N^{3s}\|P_N \bfu_1\|_{\bff_{N}(T)}\|P_N \mb{\til u}\|_{\bff_{N}(T)} \nt
\end{align}
for some $\theta\in (0, 1]$, $\del>0$ and all $t\in [0, T]$ if $s>\fr{d+1}{2}$.
Here, we note that the condition $K_1=K_2=\infty$ is assumed in this lemma.
Thus, the second and last terms on the right-hand sides of \eqref{diff1} and \eqref{diff2}, the third, fourth, and last terms on the right-hand side of \eqref{diff3} and the last term on the right-hand sides of \eqref{diff4}--\eqref{diff6}, \eqref{diff8}, and \eqref{diff9} vanish.
In this proof, we only consider the estimate of the first term on the right-hand side of \eqref{diff1}.
The first term on the right-hand sides of \eqref{diff2} and the first and second terms on the right-hand side of \eqref{diff3} are estimated by the same argument.
Also, we can obtain the estimates of other terms by the same argument as the proof of Lemma \ref{energy1} in Section 6.
From Proposition \ref{trilin1}, we have
\begin{align}
&N^{2s}\Bl|\int_0^T\int_{\R^d}(\na \cd P_N w_1)P_{N_2}\til v \cd \ove{ P_N \til u}\,dxdt\Br| \nt\\
&\les  T^{\theta}N_2^{(d-1)/2}N^{2s+1}\|P_{N}w_1\|_{F_{N, \ga}(T)}\|P_{N_2}\til v\|_{F_{N, \be}(T)}\|P_{N}\til u\|_{F_{N, \al}(T)} \nt\\
&\les  T^{\theta}N_2^{-s+\fr{d+1}{2}}\|P_{N_2}\til v\|_{F_{N, \be}(T)} N^{3s}\|P_{N}w_1\|_{F_{N, \ga}(T)}\|P_{N}\til u\|_{F_{N, \al}(T)} \nt
\end{align}
for $N_2\les N$.
Since $ \sum_{N_2\les N}N_2^{-s+\fr{d+1}{2}}\|P_{N_2}\til v\|_{F_{N, \be}(T)}\les \|\til v\|_{F_{\be}^0(T)}$, the first term on the right-hand side of \eqref{diff1} is bounded by the last term on the right-hand side of \eqref{pr_dene2_3}.
This concludes the proof of Lemma \ref{dene2}.

\subsection{Proof of Lemma \ref{dene1}}
Lemma \ref{coer2}, \eqref{diff10}, and \eqref{diff11} yield \eqref{pr_dene2_1} with $s=0$.
By Lemma \ref{s_0norm}, the second term on the right-hand side of \eqref{pr_dene2_1} is bounded by $N^{-1}T(\|\bfu_1\|_{\bff^s(T)}^3 + \|\bfu_2\|_{\bff^s(T)}^3)\|P_N \til \bfu\|_{\bff_N(T)}^2$.
Also, from Proposition \ref{trilin1}, the last term on the right-hand side of \eqref{pr_dene2_1} is evaluated by $T(\|\bfu_1\|_{\bff^s(T)} + \|\bfu_2\|_{\bff^s(T)})\|P_N \bfu\|_{\bff_N(T)}^2$.
Thus, it suffices to show
\begin{align}
&\Bl|\int_0^T(\tm{R.H.S. of }\eqref{diff1}\tm{--}\eqref{diff9})\, dt\Br| \label{pr_dene1_1}\\
&\les  T^{\theta}(\|(\bfu_1, \bfu_2)\|_{[\bff^s(T)]^2} +\|(\bfu_1, \bfu_2)\|_{[\bff^s(T)]^2}^2) \Bl(\sum_{N_1\sim N}\|P_{N_1} \mb{\til u}\|_{\bff_{N_1}(T)}\Br)^2 \nt\\
&  + T^{\theta}\|\mb{\til u}\|_{\bff^0(T)} (1+\|(\bfu_1, \bfu_2)\|_{[\bff^s(T)]^2}) \nt\\
&\hs{15mm}\times \sum_{N_1\sim N}\|P_{N_1} \mb{\til u}\|_{\bff_{N_1}(T)}\sum_{N_2\sim N}N_2^{s} \|(P_{N_2}\bfu_1, P_{N_2}\bfu_2)\|_{[\bff_{N_2}(T)]^2} \nt\\
& + T^{\theta} (\|(\bfu_1, \bfu_2)\|_{[\bff^s(T)]^2} +\|(\bfu_1, \bfu_2)\|_{[\bff^s(T)]^2}^2)\|\mb{\til u}\|_{\bff^{0}(T)}\cd  N^{-\del}\|P_N \mb{\til u}\|_{\bff_{N}(T)} \nt\\
&+  T^{\theta}K_1^{-(s-\fr{1}{2}(d+1))}\|\bfu_1\|_{\bff^s(T)}(1 +\|(\bfu_1, \bfu_2)\|_{[\bff^s(T)]^2}) \nt\\
&\hs{15mm}\times \sum_{N_1\sim N}\|P_{N_1} \mb{\til u}\|_{\bff_{N_1}(T)}\sum_{N_2\sim N}N_2^{s}\|P_{N_2} \bfu_1\|_{\bff_{N_2}(T)}  \nt\\
&+  T^{\theta}K_1^{-(s-\fr{1}{2}(d+1))}\|\bfu_1\|_{\bff^s(T)}^2(1 +\|(\bfu_1, \bfu_2)\|_{[\bff^s(T)]^2})\cd N^{-\del}\|P_N \mb{\til u}\|_{\bff_{N}(T)} \nt
\end{align}
for some $\theta\in (0, 1]$, $\del>0$ and all $T\in (0, 1]$ if $s>\fr{d+1}{2}$, $K_1\le K_2$.

In the following, we abbreviate ``right-hand side'' as ``R.H.S.''
We only consider the estimates of the fourth and sixth terms on the R.H.S. of \eqref{diff1} and the fifth and sixth terms on the R.H.S. of \eqref{diff5}.
The first, third, and fifth terms on the R.H.S. of \eqref{diff1}, the first, third, fourth, and fifth terms on the R.H.S. of \eqref{diff2}, the first, second, fifth, sixth, and seventh terms on the R.H.S. of \eqref{diff3}, the the first term on the R.H.S.s of \eqref{diff4} and \eqref{diff7}, and the first and second terms on the R.H.S.s of \eqref{diff5} and \eqref{diff8} are bounded by the similar argument as the fourth term on the R.H.S. of \eqref{diff1}.
Also, the second term on the R.H.S. of \eqref{diff1} and the second and sixth terms on the R.H.S. of \eqref{diff2}, the third, fourth, and last terms on the R.H.S. of \eqref{diff3} are estimated by the similar argument as sixth terms on the R.H.S. of \eqref{diff1}.
Moreover, the second, third, and fourth terms on the R.H.S. of \eqref{diff4}, the third, fourth, and fifth terms on the R.H.S.s of \eqref{diff5} and \eqref{diff8}, the first, second, and third terms on the R.H.S.s of \eqref{diff6} and \eqref{diff9}, and the second term on the R.H.S. of \eqref{diff7} are similar to the fifth term on the R.H.S. of \eqref{diff5}.
Finally, the last term on the R.H.S.s of \eqref{diff4}, \eqref{diff6}, \eqref{diff8}, and \eqref{diff9} are similar to the sixth term on the R.H.S. of \eqref{diff5}.

Also, we omit $J_{\le K_1}, J_{\le K_2}$ since they are harmless.
First, we consider the estimates of the fifth and sixth terms on the R.H.S. of \eqref{diff8}.
On the fifth term, we have from Lemma \ref{qu2} and the H\"older inequality that
\begin{align}
&\Bl|\int_0^T\int_{\R^d}(u_1\cd \ove{ P_{N}((\ove{\na\cd w_1})\til u)})\ove{\na\cd \Del^{-1} P_{N}\til w}\,dxdt\Br| \nt \\
&\les T\|u_1\|_{L_T^\infty L_x^\infty} \|P_N((\ove{\na\cd w})\til u)\|_{L_T^\infty L_x^2}N^{-1}\|P_N \til w\|_{L_T^\infty L_x^2} \nt\\
&\les T\|\bfu_1\|_{\bff^s(T)}\cd\Bl(\sum_{N_1\sim N} N_1^s\|P_{N_1}\bfu\|_{\bff_{N_1}(T)}+ N^{-(s-\fr{d}{2})}\|\bfu_1\|_{\bff^s(T)}\Br)\|\til \bfu\|_{\bff^0(T)} \nt\\
&\quad \times \|P_N\til \bfu\|_{\bff_N(T)} \nt
\end{align}
if $s>\fr{d}{2}$.
Thus, this term is bounded by the second and third term on the R.H.S. of \eqref{pr_dene1_1}.
On the sixth term, we have from Lemma \ref{qu1} with $\til s=s$ and Remark \ref{rem_del} that
\begin{align}
&\Bl|\int_0^T\int_{\R^d}(u_1\cd \ove{P_{N}J_{(K_1, K_2]}((\ove{\na\cd w_1})u_1)})\ove{\na\cd \Del^{-1} P_{N}\til w}\,dxdt'\Br| \nt \\
&\les T\|u_1\|_{L_T^\infty L_x^\infty} K_1^{-s}N^{s}\|P_N((\ove{\na\cd w_1})u_1)\|_{L_T^\infty L_x^2}N^{-1}\|P_N \til w\|_{L_T^\infty L_x^2} \nt\\
&\les TK_1^{-s}\|\bfu_1\|_{\bff^s(T)}\cd\Bl(\sum_{N_1\sim N}N_1^{\til s}\|P_{N_1}\bfu_1\|_{\bff_{N_1}(T)} + N^{-1/2}\|\bfu_1\|_{\bff^{s}(T)}\Br)\|\bfu_1\|_{\bff^s(T)} \nt\\
&\quad \times \|P_N\til \bfu\|_{\bff_N(T)} \nt
\end{align}
if $s\ge \fr{1}{2}(d+1)$, which yields the desired bound.


Next, we consider the estimates of the fourth and sixth terms on the R.H.S. of \eqref{diff1}.
We perform the decomposition as \eqref{pr_ene6} for both terms.
On the sixth term,  Proposition \ref{trilin1}, Remark \ref{rem_del}, and the same calculation as \eqref{pr_ene6.5} yield that
\begin{align}
&\sum_{N_1\sim N_2\gg N}\Bl|\int_0^T\int_{\R^d} P_N((\na\cd  P_{N_1}w_1) P_{N_2}v_1) \cd \ove{P_N J_{(K_1, K_2]}\til u}\,dxdt\Br|  \label{pr_dene1_2}\\
&\les T^{\theta}K_1^{-\vep}\hs{-4mm}\sum_{N_1\sim N_2\gg N}\hs{-2mm}N_1^{s+\vep}N^{-\del}\|P_{N}\til \bfu\|_{\bff_N(T)}\|P_{N_2}\bfu_1\|_{\bff_{N_2}(T)}\|P_{N_1}\bfu_1\|_{\bff_{N_1}(T)} \nt
\end{align}
if $s>\fr{1}{2}(d+1)$, $\del =s-\fr{1}{2}(d+1)$, and $\vep\ge 0$.
Thus, this part is bounded by the last term on the R.H.S. of \eqref{pr_dene1_1} if $\vep\le s$.
Also, the same argument as \eqref{pr_ene6.5}--\eqref{pr_ene8} and Remark \ref{rem_del} yield
\begin{align}
&\Bl|\int_0^T\int_{\R^d}   ([P_N, P_{N_2}v_1]\na\cd P_{N_1}w_1) \cd \ove{P_N J_{(K_1, K_2]}\til u}\,dxdt\Br| \label{pr_dene1_3}\\
&\les T^{\theta}K_1^{-\vep}N^{-1+\vep}\min\{N_1, N_2\}^{\fr{d-1}{2}}N_1 N_2 \nt\\
&\quad \times\|P_{N}\til \bfu\|_{\bff_N(T)}\|P_{N_2}\bfu_1\|_{\bff_{N_2}(T)}\|P_{N_1}\bfu_1\|_{\bff_{N_1}(T)}  \nt
\end{align}
in the case of $N_1, N_2\les N$ if $\vep\ge 0$.
Hence, the summation of the R.H.S. of \eqref{pr_dene1_3} over $N_1\les N_2\sim N$ and $N_2\ll N_1\sim N$ is bounded by the fourth term on the R.H.S. of \eqref{pr_dene1_1}.

On the fourth term on the R.H.S. of \eqref{diff1}, if $s> \fr{1}{2}(d+1)$, then we have \eqref{pr_dene1_2} and \eqref{pr_dene1_3} with $\vep=0$, $\del = s-\fr{1}{2}(d+1)$, and $P_{N_1} w_1$, $P_{N_1}\bfu_1$ replaced by $P_{N_1} \til w$, $P_{N_1}\til \bfu$, respectively.
Thus, this term is bounded by the first, second, and third terms on the R.H.S. of \eqref{pr_dene1_1}.
We note that the fifth, sixth, and seventh terms on the right-hand side of \eqref{diff3} are also bounded by the similar argument by using the decomposition as \eqref{pr_ene10}.

\appendix
\section{The optimality of the nonlinear and trilinear estimates}
In this part, we state about the optimalities of the estimates obtained in Lemma \ref{nonlin1} and Proposition \ref{trilin1} which are shown in subsection 4.2 and 5.1, respectively.
\subsection{The optimality of Lemma \ref{nonlin1}}
We state about the optimality with respect to the setting of the length of the time interval in the defintion of the function spaces defined in subsection 2.2 above.
We can not control the High$\times$Low$\to$High interaction if we change $\eta_0(NT^{-1}(t-t_N))$ into $\eta_0(N^{a}T^{-1}(t-t_N))$ with $0<a<1$ in the definition of $F_{N, \si, T}$.
For $\si\in\R\sm\{0\}$, $N\in 2^{\N_0}$, $T>0$ and $0<a$, we define
\begin{align*}
F_{N, \si, T}^a := \bl\{&u \in C_0(\R; L_x^2(\R^d))\ \Big|\ \widehat{u}(\tau, \xi) \text { is supported in } \R \times I_N,\\
&\|u\|_{F_{N, \si, T}^a}=\sup_{t_N \in \R}\|\cf[\eta_0(N^aT^{-1}(t-t_N)) \cd u]\|_{X_{N, \si}}<\infty\br\}.
\end{align*}
We also define $F_{N, \si}^a(T)$ and $F_{\si}^{s, a}(T)$ by the same manner as the definition of $F_{N, \si}(T)$ and $F_{\si}^{s}(T)$.
Then, the following lemma holds.
\begin{lem}
Let $d\in \N$, and $\al, \be, \ga\in\R\sm\{0\}$ satisfy $\al-\ga=0$ and $\be+\ga\neq 0$.
Then, for any $0<a< 1$, $s\in \R$, $T\in (0, 1]$, $1\le j \le d$, and $C>0$, there exist $f\in F_{\al}^{s, a}(T)$ and $g\in F_{\be}^{s, a}(T)$ such that
\begin{equation*}
\sup_{t\in [0, T]}\Bl\|\int_0^t e^{i(t-t')\ga\Del}\pa_{x_j}(f\bar g)(t')\,dt'\Br\|_{H^s(\R^d)}
\ge C\|f\|_{F_{\al}^{s, a}(T)}\|g\|_{F_{\be}^{s, a}(T)}.
\end{equation*}
\end{lem}
\begin{proof}
We may assume $j=1$.
Let $\del\in (1-a, 2(1-a))$ be chosen later and $M:=-(\be+\ga)/(2\ga)$.
For $k\in 2^{\N}$, we set $K=k^d$ and
\begin{gather}
D_1:=[K, K+3K^{-\del}]\times[0, 1]^{d-1}, \nt\\
D_2:=[K^{-\del}, 2K^{-\del}]\times[0, 1/2]^{d-1}, \nt\\
D:=[K-K^{-\del}, K+K^{-\del}]\times[0, 1/2]^{d-1}. \nt
\end{gather}
We also set $\cf_{x}[f_0]:= \mb{1}_{D_1}$ and $\cf_{x}[g_0]:= \mb{1}_{D_2}$.
Then, we have $(\cf_{x}[f_0]\ast \cf_{x}[\ove{g_0}])(\xi)\ge 2^{-(d-1)}K^{-\del}\mb{1}_{D}(\xi)$ for $\xi\in \R^d$.
From $\al-\ga=0$ and $M=-(\be+\ga)/(2\ga)$, we have
\[
\al|\xi-\eta|^2 - \be|\eta|^2 - \ga|\xi|^2=-2\ga(\xi-\eta-M\eta)\cd \eta.
\]
Thus, if $t'\in [0, TK^{-a}]$, $\xi\in D$, and $-\eta\in D_2$, then we have
\[
|t'(\al|\xi-\eta|^2 - \be|\eta|^2 - \ga|\xi|^2)|\les TK^{1-a-\del}\ll 1
\]
for $K\gg 1$.
Hence, by setting $t_0:=TK^{-a}$, we obtain
\[
\Re \int_0^{t_0}e^{it'(\al|\xi-\eta|^2-\be|\eta|^2-\ga|\xi|^2)}\,dt'\gtr t_0
\]
for $K\gg 1$.
Therefore, we have
\begin{align*}
&\Bl\|\int_0^{t_0} e^{i(t_0-t')\ga\Del}\pa_{x_1}((e^{it'\al\Del}f_0)( \ove{e^{it'\be\Del}g_0}))\,dt'\Br\|_{H^s(\R^d)} \\
&\gtr t_0 \|(1+|\xi|^2)^{s/2}\xi_1 (\cf_{x}[f_0]\ast \cf_{x}[\ove{g_0}])(\xi)\|_{L^2} \\
&\gtr t_0 K^{-\del}\|(1+|\xi|^2)^{s/2}\xi_1 \mb{1}_{D}(\xi)\|_{L^2} \\
&\sim TK^{s+1-3\del/2 -a} \\
&\gtr TK^{1-a -\del/2}\|e^{it\al\Del}f_0\|_{F_{\al}^{s, a}(T)}\|e^{it\be\Del}g_0\|_{F_{\be}^{s, a}(T)}.
\end{align*}
By letting $\del=3(1-a)/2$ and taking sufficiently large $K$, we finish the proof.
\end{proof}

\subsection{The optimality of Proposition \ref{trilin1}}
We remark that the exponent of $N_3^*$ in \eqref{trilin1_2} is optimal.
\begin{lem}
Let $d\in \N$ and $\al, \be, \ga\in \R\sm\{0\}$ satisfy $\al-\ga=0$, $\be+\ga\neq 0$.
We set $M:=-(\be+\ga)/(2\ga)$.
Then, for any $s<\fr{1}{2}(d-1)$, $T\in (0, 1]$, and $C>0$, there exist $u\in F_{N_1, \al}(T)$, $v\in F_{N_2, \be}(T)$, and $w\in F_{N_3, \ga}(T)$ such that
\begin{align*}
&\Bl|\int_\R\int_{\R^d}(\ove{g_{0}(t)u})(g_{0}(t)v)(g_{0}(t)w)\,dxdt\Br| \\
&\ge C(N_1^*)^{-1} (N_3^*)^{s}\|u\|_{F_{N_1, \al}(T)}\|v\|_{F_{N_2, \be}(T)}\|w\|_{F_{N_3, \ga}(T)}.
\end{align*}
where $g_0(t)=\mb{1}_{[0, T/N_1^*]}(t)$.
\end{lem}
\begin{proof}
For $K\in 2^{\N}$, we set $\til K:=K^{100d}$.
First, we prove the multi-dimensional case.
Let $\til C=\til C(\al, \be, \ga)\gg 1$.
Also, we define
\begin{align*}
D_1:&=[\til K +1/\til C^, \til K + 2/\til C]\times [3K/2, 2K]^{d-1}, \\
D_2:&=[0,1/\til C]\times [K/2, K]^{d-1}, \\
D_3:&=[\til K, \til K + 2/\til C]\times [K/2, 3K/2]^{d-1}
\end{align*}
and $\cf_x[f_j](\xi):=\mb{1}_{D_j}(\xi)$ for $j=1, 2, 3$.
From $\al=\ga$ and $M=-(\be+\ga)/(2\ga)$, we have
\[
\al|\xi_1|^2-\be|\xi_1-\xi_3|^2-\ga|\xi_3|^2=-2\ga(\xi_1-M(\xi_1-\xi_3))\cd(\xi_1-\xi_3).
\]
Since $\til C=\til C(\al, \be, \ga)\gg 1$, if $t'\in [0, T\til K^{-1}]$, $\xi_j\in D_j$ for $j=1,2,3$, and $\xi_1-\xi_2-\xi_3=0$, then we have
\[
|t'(\al|\xi_1|^2-\be|\xi_2|^2-\ga|\xi_3|^2)|\les 1/\til C\ll 1.
\]
In particular, we have
\[
\inf_{\xi_j\in D_j, \xi_1-\xi_2-\xi_3=0}\Re\int_\R g_0(t)e^{it(\al|\xi_1|^2-\be|\xi_2|^2-\ga|\xi_3|^2)}dt\gtr \til K^{-1}.
\]
Also, we have $\mb{1}_{D_2}*\mb{1}_{D_3}(\xi)\gtr K^{d-1}\mb{1}_{D_1}(\xi)$ for any $\xi\in \ope{supp}\cf_x[f_1]$.
Therefore, we obtain
\begin{align*}
\int_{\xi_1-\xi_2-\xi_3=0}\ove{\cf_x[f_1](\xi_1)}\cf_x[f_2](\xi_2)\cf_x[f_3](\xi_3)\,d\xi_1d\xi_2d\xi_3
\gtr K^{2(d-1)},
\end{align*}
and thus we have
\[
\Bl|\int_\R\int_{\R^d}g_0(t)(\ove{e^{it\al\Del}f_1})(e^{it\be\Del}f_2)(e^{it\ga\Del}f_3)\,dxdt\Br|
\gtr \til K^{-1}K^{2(d-1)}.
\]
Let $u= e^{it\al\Del}f_1$, $v= e^{it\be\Del}f_2$, $w= e^{it\ga\Del}f_3$.
Then, we have $u\in F_{{N_1, \al}}(T)$, $v\in F_{{N_2, \be}}(T)$, and $w\in F_{{N_3, \ga}}(T)$ with $N_1=N_3=\til K$ and $N_2=K$.
Also, we obtain $\|u\|_{F_{N_1, \al}(T)}\|v\|_{F_{N_2, \be}(T)}\|v\|_{F_{N_3, \ga}(T)}\sim K^{3(d-1)/2}$.
Thus, if $s<\fr{1}{2}(d-1)$, then we finish the proof by taking sufficiently large $K$.

Next, we prove the one-dimensional case.
We consider the case that (B) in Lemma \ref{modest2} does not hold.
As Lemma \ref{modest2}, let $b\in \R\sm\{-2, -1\}$ satisfy $\be=\ga(b+1)$.
We first consider the case of $b>0$.
Let $n=n(\al, \be, \ga)\in \N$ be chosen later and $C_2=C_2(\al, \be, \ga), C_3=C_3(\al, \be, \ga)\gg 1$ satisfy $C_3\le |b|C_2/2$.
For $K\in 2^{\N}$, we set
\begin{align*}
D_1:&=[(b/2+1)K+1/C_2, (b/2+1)K + b/(2C_3),] \\
D_2:&=[K, K+1/C_2], \\
D_3:&=[bK/2, bK/2 + b/(2C_3)].
\end{align*}
We also set $f_1, \ldots, f_3$, and $u, v, w$ as the multi-dimensional case.
We have
\[
\fr{b}{2}\fr{1}{1+(C_2K)^{-1}}-\fr{b}{2}\le \fr{\xi_3}{\xi_2}-\fr{b}{2}\le \fr{b}{2}\fr{1}{1+(C_3K)}-\fr{b}{2}
\]
for any $\xi_2\in D_2$ and $\xi_3\in D_3$.
In particular, it holds that $|\fr{\xi_3}{\xi_2}-\fr{b}{2}|\le \max\{1/C_2, 1/C_3\}b/(2K)$.
Thus, if $t'\in [0, TK^{-1}]$, $\xi_j\in D_j$ for $j=1,2,3$, and $\xi_1-\xi_2-\xi_3=0$, then $C_2, C_3\gg 1$ and \eqref{modest2_1} with $\tau_1-\tau_2-\tau_3=0$ yield
\[
|t'(\al|\xi_1|^2-\be|\xi_2|^2-\ga|\xi_3|^2)|\ll 1,
\]
which yields
\[
\inf_{\xi_j\in D_j, \xi_1-\xi_2-\xi_3=0}\Re\int_\R g_0(t)e^{it(\al|\xi_1|^2-\be|\xi_2|^2-\ga|\xi_3|^2)}dt\gtr  K^{-1}.
\]
From the definitions of $D_1, \ldots, D_3$, and $C_3\le |b|C_2/2$, we obtain $\mb{1}_{D_2}*\mb{1}_{D_3}\gtr \mb{1}_{D_1}$.
Therefore, a similar argument as the proof of the multi-dimensional case yields
\begin{align*}
\int_{\xi_1-\xi_2-\xi_3=0}\ove{\cf_x[u_0](\xi_1)}\cf_x[v_0](\xi_2)\cf_x[w_0](\xi_3)
&= c\int_{\R^d}\mb{1}_{D_1}(\xi_1)(\mb{1}_{D_2}*\mb{1}_{D_3})(\xi_1)\,d\xi_1 \\
&\gtr K,
\end{align*}
which yields
\[
K\Bl|\int_\R\int_{\R^d}(\ove{g_{0}(t)u})(g_{0}(t)v)(g_{0}(t)w)\,dxdt\Br|
\gtr 1.
\]
Since $\|u\|_{F_{N_1, \al}(T)}\|v\|_{F_{N_2, \be}(T)}\|w\|_{F_{N_3, \ga}(T)}\sim 1$ and $s<0$, we finish the proof by taking sufficiently large $K$.
In the case of $b<0$, we can prove this by changing the sign of $C_3$ from $+$ into $-$ in the definitions of $D_1, D_3$.
In the case of $b=0$, we set $D_1, \ldots, D_3$ by
\begin{align*}
D_1:&=[2K+1/C_2, 2K + 1/C_3], \\
D_2:&=[K, K+1/C_2], \\
D_3:&=[K, K+1/C_3],
\end{align*}
where $C_2=C_2(\al, \be, \ga), C_3=C_3(\al, \be, \ga)\gg1$ satisfy $C_2>C_3$.
\end{proof}

\section*{Acknowledgement}
The author would like to express his deep gratitude to Mamoru Okamoto for helpful discussions and comments.
The author is grateful to Takamori Kato for very valuable comments about modified energies.
The author would like to show my appreciation to Soichiro Katayama for carefully proofreading the manuscript.
This work was supported by JST SPRING, Grant Number JPMJSP2138. 

\bibliographystyle{amsplain}

\begin{thebibliography}{10}


\bibitem{CCT}
M.~Christ, J.~Colliander, and T.~Tao, \emph{A priori bounds and weak solutions for the nonlinear {S}chr\"{o}dinger equation in {S}obolev spaces of negative order}, J. Funct. Anal. \textbf{254} (2008), no.~2,  368--395. 


\bibitem{CC2004}
M.~Colin and T.~Colin, \emph{On a quasilinear {Z}akharov system describing laser-plasma interactions}, Differential Integral Equations \textbf{17}  (2004), no.~3-4, 297--330. 

  


\bibitem{GTV}
J.~Ginibre, Y.~Tsutsumi, and G.~Velo, \emph{On the Cauchy problem for the Zakharov system}, J.
 Funct. Anal. \textbf{151} (1997),
  384--436.

\bibitem{GV}
J.~Ginibre and G.~Velo, \emph{Smoothing properties and retarded estimates for
  some dispersive evolution equations}, Comm. Math. Phys. \textbf{144} (1992),
  no.~1, 163--188.
  
\bibitem{Guo}
Z.~Guo, \emph{Local well-posedness and a priori bounds for the modified Benjamin-Ono equation}, Adv. Differential Equations \textbf{16} (2011), no.~11--12, 1087--1137.

\bibitem{GO}
Z.~Guo, T.~Oh, \emph{Non-existence of solutions for the periodic cubic NLS below $L^2$}, Int. Math. Res. Not. \textbf{2018} (2016), no.~6, 1656--1729.




\bibitem{Hirayama2014}
H.~Hirayama, \emph{Well-posedness and scattering for a system of
  quadratic derivative nonlinear {S}chr\"{o}dinger equations with low
  regularity initial data}, Commun. Pure Appl. Anal. \textbf{13} (2014), no.~4,
  1563--1591. 

\bibitem{HK2019}
H.~Hirayama and S.~Kinoshita, \emph{Sharp bilinear estimates and its
  application to a system of quadratic derivative nonlinear {S}chr\"{o}dinger
  equations}, Nonlinear Anal. \textbf{178} (2019), 205--226. 

\bibitem{HKO2020}
H.~Hirayama, S.~Kinoshita, and M.~Okamoto, \emph{Well-posedness for a system of quadratic derivative nonlinear
  {S}chr\"{o}dinger equations with radial initial data}, Ann. Henri
  Poincar\'{e} \textbf{21} (2020), no.~8, 2611--2636. 

\bibitem{HKO2021}
H.~Hirayama, S.~Kinoshita, and M.~Okamoto, \emph{Well-posedness for a system of quadratic derivative nonlinear
  {S}chr\"{o}dinger equations in almost critical spaces}, J. Math. Anal. Appl.
  \textbf{499} (2021), no.~2, Paper No. 125028, 29. 

\bibitem{HKO2022}
H.~Hirayama, S.~Kinoshita, and M.~Okamoto, \emph{A remark on the well-posedness for a system of quadratic
  derivative nonlinear {S}chr\"{o}dinger equations}, Commun. Pure Appl. Anal.
  \textbf{21} (2022), no.~10, 3309--3334. 

\bibitem{HKO2024}
H.~Hirayama, S.~Kinoshita, and M.~Okamoto, \emph{{Well-posedness
  and ill-posedness for a system of periodic quadratic derivative nonlinear
  Schr\"{o}dinger equations}}, Pure Appl. Anal.
    \textbf{7} (2025), no.~2, 359--412. 

\bibitem{IKT}
A.~D. Ionescu, C.~E. Kenig, and D.~Tataru, \emph{Global well-posedness of the
  {KP}-{I} initial-value problem in the energy space}, Invent. Math.
  \textbf{173} (2008), no.~2, 265--304. 


\bibitem{KeTa}
M.~Keel and T.~Tao, \emph{Endpoint {S}trichartz estimates}, Amer. J.
  Math. \textbf{120} (1998), no.~5, 955--980. 



\bibitem{KoTa}
H.~Koch and D.~Tataru, \emph{A priori bounds for the 1{D} cubic {NLS} in negative {S}obolev spaces}, Int. Math. Res. Not. IMRN (2007), no.~16, Art.  ID rnm053, 36. 

  
\end{thebibliography}

\end{document}